\documentclass[11pt,letterpaper]{article}
\usepackage[letterpaper,margin=1in]{geometry}
\usepackage{mathrsfs}
\usepackage[T1]{fontenc}
\usepackage{lmodern}
\usepackage{amsmath,amssymb,amsthm,mathtools,bm}
\usepackage{enumitem}
\usepackage[hidelinks]{hyperref}
\usepackage{booktabs}
\usepackage{multirow}
\usepackage{array}
\usepackage{graphicx}
\allowdisplaybreaks
\numberwithin{equation}{section}
\mathtoolsset{showonlyrefs}
\usepackage{tabularray}
\UseTblrLibrary{booktabs}
\theoremstyle{definition}
\newtheorem{assumption}{Assumption}[section]
\newtheorem{definition}[assumption]{Definition}
\newtheorem{theorem}[assumption]{Theorem}
\newtheorem{lemma}[assumption]{Lemma}
\newtheorem{corollary}[assumption]{Corollary}
\newtheorem{proposition}[assumption]{Proposition}
\newtheorem{remark}[assumption]{Remark}
\newtheorem{example}[assumption]{Example}
\newenvironment{namedassumption}[2][]
{\par\medskip\noindent\textbf{Assumption \textup{[#2]}}#1\textbf{.}\ }
{\par\medskip}
\newenvironment{namedsetting}[1]
{\par\medskip\noindent\textbf{Setting \textup{[#1]}.}\ }
{\par\medskip}

\newcommand{\E}{\mathbb E}

\newcommand{\R}{\mathbb R}

\newcommand{\Linftyminus}{L^{\infty\mathrm{-}}}

\newcommand{\Cov}{\operatorname{Cov}}
\newcommand{\vol}{\operatorname{vol}}
\newcommand{\supp}{\operatorname{supp}}

\newcommand{\calB}{\mathcal B}
\newcommand{\dd}{\mathrm d}
\newcommand{\braket}[1]{\left\langle #1\right\rangle}
\newcommand{\frakL}{\mathfrak L}
\newcommand{\calP}{\mathcal P}

\title{Invariant Measures as Estimators:\\
Second-Order Stochastic Expansions and Bias Reduction}
\author{Shota Yano}
\date{}

\begin{document}
\maketitle

\begin{abstract}
We develop a general second-order asymptotic framework for estimators
constructed from invariant probability measures of data-dependent Markov
processes. The framework yields a universal bias correction for smooth
functionals, computed from the same invariant measure used to construct the
estimator, without the resampling and repeated estimation involved in
bias-correction methods such as the bootstrap and jackknife.

The statistical input is represented by a random one-form, allowing
differentials of log-likelihoods and general random criteria, as well as
estimating functions, to be treated in a common framework. Under local
regularity and localization conditions, we derive second-order stochastic
expansions for averages with respect to these invariant measures and for
decision rules defined through general loss functions. The expansions
separate the contributions of the statistical input, temperature, drift,
and loss, using precontrast geometry and an O-derivative induced by the
loss. The framework includes maximum likelihood and posterior-based
estimators and recovers existing bias-reduction methods based on the choice
of prior or adjustments to estimating equations.

In the likelihood setting, for the Jeffreys posterior $\mu_n$ and its
Fisher--Rao Fr\'echet mean $\widehat z_n$, the corrected estimator
$2\gamma(\widehat z_n)-\mu_n(\gamma)$ has frequentist bias $o(a_n^2)$ for
every fixed smooth functional $\gamma$, where $a_n$ denotes the estimation
rate. For coordinate squared loss, the theory yields a diffusion determined
only by the likelihood and Fisher information. The mean of its invariant
probability measure estimates the parameter with bias $o(a_n^2)$. This
construction remains available even when no bias-reducing prior exists.

Numerical results in the gamma shape--scale model illustrate this bias
reduction.
\end{abstract}

\section{Introduction}
\label{sec:introduction}

Consider a sequence of statistical experiments \((\mathcal X_n,\mathcal A_n,\{P_{n,\theta}:\theta\in\Theta\}),n\ge1,
\) where $\Theta\subset\R^{\mathrm p}$ is a compact convex domain with smooth
boundary. Assume that the true parameter $\theta_0$ lies in the interior
$\Theta^\circ$ of $\Theta$, and write
$\E_{n,\theta}$ for expectation under $P_{n,\theta}$. Suppose that, for each
$n$, the family $\{P_{n,\theta}:\theta\in\Theta\}$ is dominated by a
reference measure $\nu_n$, with density \(  p_{n,\theta}(x_n)
  :=  \frac{\dd P_{n,\theta}}{\dd\nu_n}(x_n)\). Assume that $p_{n,\theta}(x_n)$ is strictly positive and smooth in
$\theta\in\Theta$ for $\nu_n$-almost every $x_n$. For the observed data
$X_n$, define the log likelihood \(  \ell_n(\theta)
  :=  \log p_{n,\theta}(X_n)\). Let $a_n\downarrow0$ denote the convergence rate; in regular i.i.d.\ models, typically $a_n=n^{-1/2}$. We assume that
\begin{equation}
  g_{ij}(\theta)
  :=
  \lim_{n\to\infty}
  a_n^2
  \E_{n,\theta}\!\left[
    \partial_i\ell_n(\theta)\,
    \partial_j\ell_n(\theta)
  \right]
  \label{eq:intro-fisher-metric}
\end{equation}
exists and defines a smooth positive-definite Riemannian metric $g$ on
$\Theta$, where $\partial_i$ denotes differentiation with respect to the
$i$-th standard coordinate. We write $(g^{ij})=(g_{ij})^{-1}$. We call the Riemannian volume measure $\vol$ induced by $g$ the
Jeffreys prior associated with $g$. 

Given a smooth strictly positive prior
density $\pi$ with respect to $\vol$, define the posterior distribution by
\[
  \Pi_n(\dd\theta)
  :=
  \frac{
    e^{\ell_n(\theta)}\pi(\theta)\,\vol(\dd\theta)
  }{
    \displaystyle
    \int_\Theta
    e^{\ell_n(\vartheta)}\pi(\vartheta)\,\vol(\dd\vartheta)
  },
\]
Under regularity conditions, the frequentist bias of the posterior mean
generally satisfies
\[
  \E_{n,\theta_0}\!\left[
    \int_\Theta \theta\,\Pi_n(\dd\theta)
  \right]
  -\theta_0
  =
  O(a_n^2).
\]
A classical approach to reducing this bias is to choose the prior density
$\pi$ so that the leading $a_n^2$-term vanishes; see, for example,
Hartigan~\cite{Hartigan1965}, Sakai, Matsuda and
Kubokawa~\cite{SakaiMatsudaKubokawa2025}, and Miyata and
Yanagimoto~\cite{MiyataYanagimoto2026}. The resulting condition can be
written as a system of first-order partial differential equations
\begin{equation}
  g^{ij}(\theta)\,\partial_j\log\pi(\theta)
  =
  V^{*i}(\theta),
  \qquad
  i=1,\ldots,\mathrm p,
  \label{eq:intro-prior-pde}
\end{equation}
where $V^*$ is a smooth vector field for which the leading
$a_n^2$ term in the coordinate bias vanishes. Hence a prior
satisfying the bias-removal condition can exist only when $V^*$ admits the
gradient representation \eqref{eq:intro-prior-pde}. In multiparameter
models this integrability requirement need not hold; Sakai, Matsuda and
Kubokawa~\cite{SakaiMatsudaKubokawa2025} give examples in which no such
prior exists. 

On the other hand, a Langevin diffusion can be constructed so that the
posterior distribution $\Pi_n$ is its invariant probability measure, and
such diffusions are widely used in MCMC. Consider the reflected Langevin diffusion
\begin{equation}
  \dd\theta_t
  =
  \frac12
  \left\{
    g^{-1}\dd\ell_n(\theta_t)
    +
    g^{-1}\dd\log\pi(\theta_t)
  \right\}\dd t
  +
  \dd B_t
  -\mathbf n(\theta_t)\,\dd K_t,
  \label{eq:intro-posterior-langevin}
\end{equation}
where $B$ is $g$-Brownian motion, $\mathbf n$ is the outward $g$-unit
normal along the boundary $\partial\Theta$, and $K$ is the boundary local time. The term
$-\mathbf n(\theta_t)\,\dd K_t$ represents the reflection that pushes the
process back into $\Theta$ when it hits the boundary. The posterior distribution $\Pi_n$ is an invariant probability measure of
\eqref{eq:intro-posterior-langevin}. Thus, the prior enters the drift through
the vector field $g^{-1}\dd\log\pi$. Equation~\eqref{eq:intro-prior-pde}
therefore suggests replacing this term directly by $V^*$, even when no prior
density $\pi$ satisfying \eqref{eq:intro-prior-pde} exists.

Replace $g^{-1}\dd\log\pi$ in
\eqref{eq:intro-posterior-langevin} by $V^*$. We verify that, when the
resulting reflected diffusion is written in standard coordinates, the process takes the form
\begin{equation}
  \dd\theta_t^i
  =
  \frac12 g^{ij}(\theta_t)\partial_j\ell_n(\theta_t)\,\dd t
  +\sigma^i{}_{\alpha}(\theta_t)\,\dd W_t^\alpha
  -\mathbf n^i(\theta_t)\,\dd K_t,
  \qquad
  \sigma\sigma^{\mathsf T}=g^{-1}.
  \label{eq:intro-natural-gradient-diffusion}
\end{equation}
Here the Einstein summation rule is used, $W=(W^1,\ldots,W^{\mathrm p})$ is
a standard $\R^{\mathrm p}$-valued Brownian motion, and $\mathbf n$ and
$K$ have the same meanings as in \eqref{eq:intro-posterior-langevin}. Let
$\mu_n$ be an invariant probability measure of
\eqref{eq:intro-natural-gradient-diffusion}. We show that, under regularity
conditions,
\begin{equation}
  \E_{n,\theta_0}\!\left[
    \int_\Theta \theta\,\mu_n(\dd\theta)
  \right]
  -\theta_0
  =
  o(a_n^2).
  \label{eq:intro-invariant-mean-bias}
\end{equation}
Unlike \eqref{eq:intro-posterior-langevin},
\eqref{eq:intro-natural-gradient-diffusion} is not constructed as a
sampling algorithm for a prescribed posterior distribution. 

This observation motivates the framework developed in this paper: we begin
with a data-dependent Markov process and develop a general framework for
constructing decision rules using an invariant probability measure of that
process. The invariant probability measure is not specified in advance, and
its explicit form need not be known. This reverses the usual MCMC
perspective, in which a target distribution (such as a posterior
distribution) is specified first and the Markov process is introduced as a
computational device. We develop a second-order asymptotic theory for
decision rules constructed in this way.

We now introduce the general setting of the paper, using the same notation
as above for the corresponding objects in this more general setting. For
each $n\ge1$, let
$(\mathcal X_n,\mathcal A_n,\{P_{n,\theta}:\theta\in\Theta\})$ be a
statistical experiment, and retain the notation $\E_{n,\theta}$ for
expectation under $P_{n,\theta}$. We now allow the parameter space $\Theta$
to be a compact connected smooth manifold with boundary, and assume that the
true parameter $\theta_0$ lies in its interior $\Theta^\circ$.

The statistical input consists of a random smooth one-form $\Psi_n$ on
$\Theta$, a deterministic sequence $a_n\downarrow0$, and a smooth
Riemannian metric $g$ on $\Theta$. In the setting above,
\begin{equation}
  \Psi_n=-a_n^2\dd\ell_n,
  \label{eq:intro-likelihood-specialization}
\end{equation}
and $g$ is the Fisher information metric defined in
\eqref{eq:intro-fisher-metric}. We refer to these specifications as the
likelihood setting. More generally, this one-form formulation
includes general random criteria, by taking $\Psi_n$ to be an appropriately
scaled differential of the criterion, as well as estimating functions after Godambe normalization; see Appendix \ref{app:godambe-normalization}.

Fix a temperature parameter $\tau\ge0$ and a smooth vector field $V$ on
$\Theta$. We consider the normally reflected process
\begin{equation}
  \dd\theta_t
  =
  \frac12\left\{
    -a_n^{-2}g^{-1}\Psi_n(\theta_t)+V(\theta_t)
  \right\}\dd t
  +\sqrt{\tau}\,\dd B_t
  -\mathbf n(\theta_t)\,\dd K_t.
  \label{eq:intro-general-process}
\end{equation}
Here $g^{-1}\Psi_n$ denotes the vector field dual to $\Psi_n$ with respect
to $g$, $B$ is $g$-Brownian motion, $\mathbf n$ is the outward $g$-unit
normal, and $K$ is the boundary local time; the reflection term is omitted
when the boundary is empty. Let $\mu_n$ denote an invariant probability
measure of \eqref{eq:intro-general-process}, and write
$\mu_n(f):=\int_\Theta f(\theta)\,\mu_n(\dd\theta)$. In the likelihood setting, the choice $(\tau,V)=(0,0)$ gives the gradient flow of
$-\ell_n$ with respect to $g$. If an MLE
$\widehat\theta_n^{\mathrm{MLE}}$ lies in $\Theta^\circ$, then
$\delta_{\widehat\theta_n^{\mathrm{MLE}}}$ is an invariant probability
measure. The choice $(\tau,V)=(1,g^{-1}\dd\log\pi)$ gives the posterior
Langevin diffusion \eqref{eq:intro-posterior-langevin}.

Let the decision space $Z$ be a compact connected smooth manifold with
boundary, and let the loss function $W:\Theta\times Z\to\R$ be continuous
such that, for every $\theta\in\Theta$, $W(\theta,\cdot)$ has a unique
minimizer \(  m(\theta)
  :=
  \operatorname*{argmin}_{z\in Z}W(\theta,z)\). We define the decision rule $\widehat z_n$ by
\begin{equation}
  \widehat z_n
  \in
  \operatorname*{argmin}_{z\in Z}
  \int_\Theta W(\theta,z)\,\mu_n(\dd\theta).
  \label{eq:intro-decision-rule}
\end{equation}
For example, if $Z=\Theta\subset\R^{\mathrm p}$ is a compact convex
domain with smooth boundary and $W(\theta,z)=|\theta-z|^2$, then
\begin{equation}
  m(\theta)=\theta,
  \qquad
  \widehat z_n
  =
  \int_\Theta \theta\,\mu_n(\dd\theta).
  \label{eq:intro-coordinate-squared-loss-rule}
\end{equation}

We first derive second-order stochastic expansions of $\mu_n(f)$ and $\gamma(\widehat z_n)$, where $f:\Theta\to\R^{\mathrm q}$ and $\gamma:Z\to\R^{\mathrm q}$ are smooth. These expansions make explicit the effects of $\tau$ and $V$ on
$\mu_n(f)$ and $\gamma(\widehat z_n)$, and of $W$ on
$\gamma(\widehat z_n)$. One consequence of these expansions is that $V$
can be chosen to control the leading terms in the frequentist biases of
$\mu_n(f)$ and $\gamma(\widehat z_n)$, relative to $f(\theta_0)$ and
$\gamma(m(\theta_0))$, respectively. Furthermore, for each $\tau>0$ and
loss $W$, we characterize a condition on $V$ under which
$\tau^{-1}\{\mu_n(\gamma\circ m)-\gamma(\widehat z_n)\}$
approximates the frequentist bias
$\E_{n,\theta_0}[\gamma(\widehat z_n)]-\gamma(m(\theta_0))$
with an $o_{L^1}(a_n^2)$ error. Consequently,
\begin{equation}
  \E_{n,\theta_0}\!\left[
    \gamma(\widehat z_n)
    -\frac1\tau
    \left\{
      \mu_n(\gamma\circ m)-\gamma(\widehat z_n)
    \right\}
  \right]
  -\gamma(m(\theta_0))
  =o(a_n^2).
  \label{eq:intro-loss-bias-correction}
\end{equation}
In particular, when $\tau=1$,
\begin{equation}
  \E_{n,\theta_0}\!\left[
    2\gamma(\widehat z_n)-\mu_n(\gamma\circ m)
  \right]
  -\gamma(m(\theta_0))
  =o(a_n^2).
  \label{eq:intro-unit-temperature-bias-correction}
\end{equation}
This closely parallels bootstrap bias correction: the leading frequentist bias is estimated and subtracted from $\gamma(\widehat z_n)$. Here, however, the bias estimate is obtained from the same invariant probability measure used to construct $\widehat z_n$, rather than by generating new datasets and recomputing the decision rule.

Particularly simple consequences arise in the likelihood setting when
$\tau=1$ and $Z=\Theta$. First, suppose that
$\Theta\subset\R^{\mathrm p}$ is a compact convex domain with smooth
boundary and $W(\theta,z)=|\theta-z|^2$. The corresponding choice of $V$
for \eqref{eq:intro-unit-temperature-bias-correction} is precisely the
vector field $V^*$ appearing in \eqref{eq:intro-prior-pde}. Taking
$\gamma=\operatorname{id}_\Theta$ and using
\eqref{eq:intro-coordinate-squared-loss-rule} in
\eqref{eq:intro-unit-temperature-bias-correction} recovers
\eqref{eq:intro-invariant-mean-bias}, while substituting $V=V^*$ into
\eqref{eq:intro-general-process} yields
\eqref{eq:intro-natural-gradient-diffusion}. Thus, both
\eqref{eq:intro-natural-gradient-diffusion} and
\eqref{eq:intro-invariant-mean-bias} are recovered as special cases of the
general theory.

Second, let $\Theta$ be a compact connected smooth manifold with boundary
and take $W(\theta,z)=d(\theta,z)^2$, where $d$ is the Riemannian distance
induced by $g$. Then $m=\operatorname{id}_\Theta$. The corresponding choice
of $V$ for
\eqref{eq:intro-unit-temperature-bias-correction} is $V=0$. In this case,
the invariant probability measure of \eqref{eq:intro-general-process} is
$\mu_n(\dd\theta)\propto e^{\ell_n(\theta)}\vol(\dd\theta)$, namely
the Jeffreys posterior, and the decision rule $\widehat z_n$ is its
Fisher--Rao Fr\'echet mean.

We next relate the present framework to existing bias-reduction methods. As seen above, at $\tau=1$, the prior-based formulation requires
$V=g^{-1}\dd\log\pi$, and hence imposes an integrability condition on $V$;
the present framework removes this restriction. At zero temperature, the invariant probability measure
may be taken to be a point mass at a root of the corresponding estimating
equation, so that $\mu_n(f)$ reduces to the usual plug-in estimator. In this
case, changing $V$ corresponds to adjusting the estimating equation. For estimation of the parameter itself
in the likelihood setting, the choice of $V$ that removes the leading bias
recovers the adjusted score of Firth~\cite{Firth1993}, which need not arise
from a penalized likelihood~\cite{KosmidisFirth2009}. For a general smooth
estimand, Hirose and Mano~\cite{HiroseMano2026} study bias reduction through
penalized likelihoods, while Kosmidis and Lunardon~\cite{KosmidisLunardon2024}
develop bias-reducing adjustments for general estimating equations. These
results can therefore be viewed as zero-temperature instances of bias
reduction for $\mu_n(f)$ within the present framework.

The second-order expansions also permit direct comparison of decision rules arising from different choices of $\tau$, $V$, and $W$. In the likelihood setting with $\Theta = Z \subset \mathbb{R}^{\mathrm{p}}$ and $W(\theta,z) =|\theta -z|^2$, the choice $(\tau,V)=(1,g^{-1}\dd\log\pi)$ yields the posterior mean, whereas $(\tau,V)=(0,0)$ recovers the MLE when the invariant measure is selected as the Dirac mass at an interior MLE. Comparing their second-order expansions therefore recovers the moment-matching problem of choosing $\pi$ so that the posterior mean and the MLE agree to second order. In regular i.i.d.\ models, Ghosh and Liu~\cite{GhoshLiu2011} study priors for which the posterior mean of the parameter and the MLE agree up to $o_P(n^{-1})$, and Tanaka~\cite{Tanaka2023} gives an information-geometric representation of the corresponding condition. Okudo and Yano~\cite{OkudoYano2026} more generally study pairs of priors under which a posterior mean and a MAP estimator agree to the same order. Sakai, Matsuda and Kubokawa~\cite{SakaiMatsudaKubokawa2025} further show that the condition for a second-order unbiased posterior mean decomposes into a Firth adjustment and a moment-matching contribution. Our comparison result allows $\tau$, $V$, and $W$ to vary simultaneously between the two decision rules, while keeping the same target map $m$, and also applies when $\Psi_n$ is not the differential of a random criterion.

The present framework is also related to the asymptotic theory of likelihood- and Bayes-type estimators. The classical theory of Ibragimov and Has'minskii~\cite{IbragimovHasminskii1981} obtained moment convergence for maximum likelihood and Bayes estimators through an analysis of the likelihood. Extending this likelihood-based analysis to general random criteria, Yoshida~\cite{Yoshida2011,Yoshida2025Simplified} developed quasi-likelihood analysis (QLA) and established moment convergence for M- and Bayes-type estimators. 
Chernozhukov and Hong~\cite{ChernozhukovHong2003} similarly construct a quasi-posterior from a general random criterion and use its means and quantiles as frequentist estimators. General loss functions for Bayes-type estimators have also been considered by Ogihara~\cite{Ogihara2019}. In contrast, our statistical input is a random one-form $\Psi_n$, which need not be the differential of a scalar random criterion, and the vector field $V$ need not be a gradient vector field. As a result, the invariant probability measure $\mu_n$ need not admit a Gibbs-type representation. We derive explicit second-order stochastic expansions for both $\mu_n(f)$ and $\gamma(\widehat z_n)$ and use them to control second-order bias. The required localization of $\mu_n$ is imposed directly as an assumption.

A distinctive feature of our second-order expansions is that they are formulated in information-geometric terms. Eguchi~\cite{Eguchi1992} showed that a contrast function on $\Theta\times\Theta$ induces a Riemannian metric and a pair of affine connections on $T\Theta$. Henmi and Matsuzoe~\cite{HenmiMatsuzoe2011} generalized this framework to a precontrast, a function on $T\Theta\times\Theta$ that abstracts the derivative of a contrast function with respect to its first argument. We use their precontrast geometry to describe the higher-order structure of the statistical input $\Psi_n$.

A further contribution of this paper is to construct a corresponding geometry from the loss $W:\Theta\times Z\to\R$. In general, $\Theta$ and $Z$ are different manifolds, and our construction yields an O-derivative in the sense of Abe~\cite{Abe1988}. This loss geometry describes the loss-dependent terms in the second-order expansion of $\gamma(\widehat z_n)$ and provides the geometric structure underlying the bias correction in \eqref{eq:intro-loss-bias-correction}.

The remainder of the paper is organized as follows. Section~\ref{sec:setup} introduces the statistical experiment, the one-form, the invariant probability measure, and the loss-based decision rule. Section~\ref{sec:precontrast-geometry-assumptions} develops the precontrast and loss geometries and states the assumptions used in the asymptotic analysis. Section~\ref{sec:main-expansion} establishes the second-order stochastic expansions of $\mu_n(f)$ and $\gamma(\widehat z_n)$, together with the resulting bias and comparison formulas. Section~\ref{sec:bias} develops the bias correction. Section~\ref{sec:simulation} presents a numerical experiment for the gamma shape--scale model. The proofs are given in Sections~\ref{sec:proof-main-expansion} and~\ref{sec:proof-remaining-results}, and the Godambe normalization is discussed in Appendix~\ref{app:godambe-normalization}.

\section{Statistical Setup}
\label{sec:setup}

\subsection{Statistical experiment and one-form}

Let $\Theta$ be a compact connected $C^\infty$ manifold with smooth boundary, let $\theta_0\in\Theta^\circ$ denote the true parameter, and set $\mathrm p:=\dim\Theta$. 

For each $n\ge1$, consider the statistical experiment
$(\mathcal X_n,\mathcal A_n,\{P_{n,\theta}:\theta\in\Theta\})$.
The statistical input is a map
\[
  \Psi_n:\mathcal X_n\times\Theta\longrightarrow T^*\Theta,
  \qquad
  \Psi_n(x_n,\theta)\in T_\theta^*\Theta,
\]
such that, for every fixed $\theta\in\Theta$, the map
\[
  (\mathcal X_n,\mathcal A_n)\ni x_n
  \longmapsto \Psi_n(x_n,\theta)\in T_\theta^*\Theta
\]
is measurable, while, for every fixed $x_n\in\mathcal X_n$, the section
$\theta\mapsto\Psi_n(x_n,\theta)$ is a smooth one-form on $\Theta$.
We usually suppress $x_n$ from the notation. Let $a_n\downarrow0$ be a
deterministic sequence, and fix a smooth Riemannian metric $g$ on $\Theta$.
The relation between $g$ and the statistical input $\Psi_n$ is specified
through Assumptions~\textup{[A1]} and~\textup{[A3]} below.

This one-form formulation includes the usual estimating-function setting.
Under the conditions stated in Appendix~\ref{app:godambe-normalization}, an
ordinary $\R^{\mathrm p}$-valued estimation function can be converted into a
random one-form by the Godambe normalization. When $\Psi_n$ is obtained in
this way, $g$ is the corresponding Godambe information.

\subsection{Generator and invariant probability measure}
\label{subsec:generator-invariant-measure}

Fix a smooth vector field $V\in\mathfrak X(\Theta)$ and $\tau\ge0$. Let
$\Delta$ denote the Laplace--Beltrami operator associated with $g$, and let
$\mathbf n$ denote the outward $g$-unit normal vector field along
$\partial\Theta$ when the boundary is nonempty. For
$x_n\in\mathcal X_n$, let
$L_n=L_n(x_n)$ be the operator
\begin{equation}
  L_n f
  :=
  \frac{\tau}{2}\Delta f
  +\frac12\left\{
    -a_n^{-2}g^{-1}\Psi_n(x_n,\cdot)+V
  \right\}(f),
  \label{eq:full-generator}
\end{equation}
acting on smooth functions $f\in C^\infty(\Theta)$ satisfying the Neumann
boundary condition $\mathbf n(f)=0$ on $\partial\Theta$. Here
$g^{-1}\Psi_n$ denotes the vector field dual to $\Psi_n$ with respect to $g$.
The normally reflected Markov process with generator $L_n$ is the process in
\eqref{eq:intro-general-process}. If $\partial\Theta=\varnothing$, the reflection term and boundary
condition are omitted. See
\cite{LionsSznitman1984,Saisho1987,Wang2010Boundary} for standard constructions
of normally reflected processes and their Neumann semigroups.

Let $\mathcal P(\Theta)$ denote the set of Borel probability measures on
$\Theta$, equipped with the weak topology. For every $n$ and the fixed
$(\tau,V)$ above, we assume that there exists a measurable map
\[
  \mathcal X_n\ni x_n
  \longmapsto \mu_n(x_n,\cdot)\in\mathcal P(\Theta)
\]
such that, for every $x_n\in\mathcal X_n$, the probability measure
$\mu_n(x_n,\cdot)$ is invariant for the process with generator
$L_n(x_n)$. We fix one such measurable selection and write $\mu_n$ for the
resulting random probability measure. If $\tau>0$, the invariant probability
measure is unique, so a nontrivial choice is needed only at $\tau=0$.
For a fixed smooth map $f:\Theta\to\R^{\mathrm q}$, write
\begin{equation}
  \mu_n(f):=\int_\Theta f(\theta)\,\mu_n(\dd\theta).
  \label{eq:estimator-definition}
\end{equation}

\subsection{Decision space and decision rule}

Let $Z$ be a compact finite-dimensional smooth manifold with possibly empty
smooth boundary, equipped with its Borel $\sigma$-field, and call $Z$ the
decision space. Fix a continuous function
\[
  W:\Theta\times Z\longrightarrow\R,
\]
which is assumed to be a loss in the sense of
Definition~\ref{def:loss} below. Given the invariant probability
measure $\mu_n$ fixed above, define
\begin{equation}
  Q_n(z):=\int_\Theta W(\theta,z)\,\mu_n(\dd\theta),
  \qquad z\in Z,
  \label{eq:integrated-loss}
\end{equation}
and choose a measurable decision rule $\widehat z_n$ satisfying
\begin{equation}
  \widehat z_n\in\operatorname*{argmin}_{z\in Z}Q_n(z).
  \label{eq:decision-rule}
\end{equation}
Compactness and continuity ensure that the set of minimizers is nonempty and
admits a measurable selection. We fix one such selection; uniqueness of the
minimizer of $Q_n$ is not assumed. If $\mu_n$ is a posterior distribution,
$\widehat z_n$ is the corresponding Bayes rule under $W$.

\section{Geometry and Assumptions}
\label{sec:precontrast-geometry-assumptions}

\subsection{Notation}
\label{subsec:tensor-contraction-notation}

Throughout, all manifolds and all vector bundles are smooth and
finite-dimensional. Manifolds may have boundary. We fix $\mathrm q<\infty$. We use $M,N$ for
manifolds.

For a manifold $M$, we also denote by $\R^{\mathrm q}$ the trivial vector
bundle $M\times\R^{\mathrm q}\to M$ with typical fiber $\R^{\mathrm q}$.
Accordingly, for a vector bundle $E\to M$, the notation
$E\otimes\R^{\mathrm q}$ means $E\otimes(M\times\R^{\mathrm q})$.

Let $E_0,E_1,\ldots,E_r\to M$ be vector bundles. For
\[
  A\in\Gamma(E_1^*\otimes\cdots\otimes E_r^*\otimes E_0),
  \qquad
  B_j\in\Gamma(E_j),\quad j=1,\ldots,r,
\]
we write
\[
  A[B_1,\ldots,B_r]\in\Gamma(E_0)
\]
for the canonical contraction of the $E_j^*$-factor against $B_j$, for each
$j$, in the indicated order. Tensor products of vector bundles may themselves
be taken as the bundles $E_j$.

For a smooth map $f:M\to N$, we regard its differential as
\[
  \dd f\in\Gamma(T^*M\otimes f^*TN).
\]
Thus $\dd f[X]\in\Gamma(f^*TN)$ denotes contraction against
$X\in\mathfrak X(M)$. In particular, if
$f\in C^\infty(M;\R^{\mathrm q})$, then under the canonical trivialization of
$T\R^{\mathrm q}$,
\[
  \dd f\in\Gamma(T^*M\otimes\R^{\mathrm q}).
\]

We use the notion of an O-derivative introduced by Abe~\cite{Abe1988}.

\begin{definition}[O-derivative]
Let $E,F\to M$ be vector bundles and let $P:E\to F$ be a bundle
homomorphism. An O-derivative with principal homomorphism $P$ is an
$\R$-linear map
\[
  D:\Gamma(E)\longrightarrow\Gamma(T^*M\otimes F)
\]
satisfying
\[
  D(as)=\dd a\otimes P(s)+aDs,
  \qquad
  a\in C^\infty(M),\quad s\in\Gamma(E).
\]
For $X\in\mathfrak X(M)$, we use the usual notation
\[
  D_Xs:=(Ds)[X]\in\Gamma(F),
\]
where the right-hand side uses the contraction notation introduced above.
\end{definition}

Ordinary connections on $E$ are precisely the O-derivatives with $F=E$ and
$P=\operatorname{id}_E$.

An O-derivative $D$ with principal homomorphism $P:E\to F$ induces a dual
O-derivative, again denoted by $D$,
\[
  D:\Gamma(F^*)\longrightarrow\Gamma(T^*M\otimes E^*),
\]
defined by
\[
  (D_X\alpha)[s]
  =
  X\!\left\{\alpha[P(s)]\right\}
  -\alpha[D_Xs],
\]
for $\alpha\in\Gamma(F^*)$, $X\in\mathfrak X(M)$, and $s\in\Gamma(E)$.
Its principal homomorphism is $P^*:F^*\to E^*$.

If $D_1,\ldots,D_r$ are O-derivatives with the same principal homomorphism
$P:E\to F$, then any affine combination $\sum_{a=1}^r c_aD_a$, with
$\sum_{a=1}^r c_a=1$, is again an O-derivative with principal homomorphism
$P$. If instead $\sum_{a=1}^r c_a=0$, then $\sum_{a=1}^r c_aD_a$ is
$C^\infty(M)$-linear and is identified with a section of
$T^*M\otimes E^*\otimes F$.

We extend O-derivatives componentwise to $\R^{\mathrm q}$-valued sections.
Namely, for
\[
  s=(s^1,\ldots,s^{\mathrm q})
  \in\Gamma(E\otimes\R^{\mathrm q}),
\]
we set
\[
  Ds:=(Ds^1,\ldots,Ds^{\mathrm q})
  \in\Gamma(T^*M\otimes F\otimes\R^{\mathrm q}).
\]

Let $E\to M$ be a vector bundle equipped with a bundle metric $h$. For
$s_1,s_2\in\Gamma(E)$, we write $\braket{s_1,s_2}_h$ for their pointwise
inner product. We identify $h$ with the bundle isomorphism
\[
  h:E\longrightarrow E^*,
  \qquad
  h(s)=\braket{s,\cdot}_h,
\]
and denote its inverse by $h^{-1}:E^*\to E$. The dual bundle metric on $E^*$
is also denoted by $h^{-1}$. We write the actions of these bundle
isomorphisms by juxtaposition, as $hs$ for $s\in\Gamma(E)$ and
$h^{-1}\alpha$ for $\alpha\in\Gamma(E^*)$.

If $h$ is a bundle metric on $TM$, hence a Riemannian metric on $M$, we write
$d_h$ for the corresponding Riemannian distance and $\nabla^h$ for its
Levi--Civita connection.

We now specialize these conventions to the parameter manifold $\Theta$,
equipped with the Riemannian metric $g$. For inner products induced by $g$ or
$g^{-1}$, we suppress the metric subscript. Thus
\[
  \braket{X,Y}:=\braket{X,Y}_g,
  \qquad
  \braket{\alpha,\beta}:=\braket{\alpha,\beta}_{g^{-1}}
\]
for $X,Y\in\mathfrak X(\Theta)$ and $\alpha,\beta\in\Omega^1(\Theta)$. We also write
\[
  d:=d_g,
  \qquad
  \nabla:=\nabla^g.
\]

For a connection $D$ on $T\Theta$ and $f\in C^\infty(\Theta;\R^{\mathrm q})$, we define
\begin{equation}
  \Delta^D f:=(D\dd f)[g^{-1}].
  \label{eq:D-Laplacian-definition}
\end{equation}
With $\nabla = \nabla^g$, the Laplace--Beltrami operator $\Delta$ introduced in Section~2 is 
\[\Delta  = \Delta^\nabla.\]
For a tensor field
\[
  T\in
  \Gamma\!\left(
    (T\Theta)^{\otimes r}\otimes
    (T^*\Theta)^{\otimes s}
  \right),
\]
we write $|T|$ for the pointwise norm induced by $g$ and $g^{-1}$, and set
\[
  \|T\|_\infty:=\sup_{\theta\in\Theta}|T(\theta)|.
\]

Finally, write $\E_{n,\theta}$ and $\Cov_{n,\theta}$ for expectation and
covariance under $P_{n,\theta}$, respectively. Let
$X_n$ be a random variable taking values in a fixed finite-dimensional vector
space. For fixed $\theta\in\Theta$, the statements
\[
  X_n=O_{\Linftyminus(\theta)}(r_n),
  \qquad
  X_n=o_{\Linftyminus(\theta)}(r_n)
\]
mean, respectively, that for every $p<\infty$,
\[
  \bigl\||X_n|/r_n\bigr\|_{L^p(P_{n,\theta})}=O(1),
  \qquad
  \bigl\||X_n|/r_n\bigr\|_{L^p(P_{n,\theta})}=o(1),
\]
where $|\cdot|$ is any norm on the underlying finite-dimensional vector
space. The definition is independent of the choice of norm.

\subsection{Precontrast geometry}
\label{subsec:precontrast-geometry}
We introduce a deterministic one-form family, called a precontrast, associated with the statistical input $\Psi_n$. Its relation to \(\Psi_n\) is specified in Assumption~\textup{[A1]} below. The notion of a precontrast and the induced connections presented in this subsection are due to Henmi and Matsuzoe~\cite{HenmiMatsuzoe2011}.

\begin{definition}[Precontrast]
\label{def:precontrast}
Suppose that, for each $\theta\in\Theta$, a smooth one-form
$\Psi^\theta\in\Omega^1(\Theta)$ is given and satisfies
\begin{equation}
  \Psi^\theta(\theta)=0.
  \label{eq:diagonal-zero}
\end{equation}
Assume further that the family
$\{\Psi^\theta\}_{\theta\in\Theta}$ defines a smooth map
$\Psi:T\Theta\times\Theta\to\R$ by
\[
  \Psi(\cdot,\theta):=\Psi^\theta,
  \qquad \theta\in\Theta.
\]
If there exists a Riemannian metric $g_\Psi$ on $\Theta$ such that
\[
  \braket{X,Y}_{g_\Psi}(\theta)
  =
  Y\!\left\{\Psi^\theta[X]\right\}(\theta)
\]
for every $X,Y\in\mathfrak X(\Theta)$ and $\theta\in\Theta$, then
$\{\Psi^\theta\}_{\theta\in\Theta}$, or simply $\Psi$, is called a
precontrast.
\end{definition}

By \eqref{eq:diagonal-zero},
$Y\{\Psi^\theta[X]\}(\theta)$ depends only on $X_\theta$ and $Y_\theta$.
Hence, for a precontrast $\Psi$, the Riemannian metric $g_\Psi$ is uniquely
determined by $\Psi$.

Throughout the remainder of the paper, we assume that
$\Psi$ is a precontrast and that the Riemannian metric $g$ fixed in
Section~\ref{sec:setup} satisfies $g=g_\Psi$.

For later use, we introduce notation for derivatives of $\Psi$ with respect
to its two arguments. For $X\in\mathfrak X(\Theta)$, set
\[\Psi(X,\cdot)(\theta_1,\theta_2):=\Psi^{\theta_2}[X](\theta_1).\]
Let $X^{(1)}$ and $Y^{(2)}$ denote the lifts to the first and second factors
of $\Theta\times\Theta$ of $X,Y\in\mathfrak X(\Theta)$, respectively. For
$X_1,\ldots,X_r,Y_1,\ldots,Y_s\in\mathfrak X(\Theta)$, with $r\ge1$ and
$s\ge0$, set
\begin{equation}
\begin{aligned}
  &\Psi(X_1\cdots X_r\mid Y_1\cdots Y_s)(\theta)
  \\
  &\qquad:=
  \left.
  X_1^{(1)}\cdots X_{r-1}^{(1)}
  Y_1^{(2)}\cdots Y_s^{(2)}
  \{\Psi(X_r,\cdot)\}
  \right|_{(\theta,\theta)}.
  \label{eq:precontrast-jet}
\end{aligned}
\end{equation}
When there is no derivative in the second argument, write
$\Psi(X_1\cdots X_r\mid)$.

Since \eqref{eq:diagonal-zero} gives
$\Psi(X,\cdot)(\theta,\theta)=0$, differentiating along the diagonal in the
direction $Y$ yields
\begin{equation}
  \Psi(YX\mid)+\Psi(X\mid Y)=0.
  \label{eq:diagonal-first-identity}
\end{equation}
Consequently,
\begin{equation}
  \braket{X,Y}
  =\Psi(YX\mid)
  =-\Psi(X\mid Y).
  \label{eq:precontrast-metric}
\end{equation}

\begin{definition}
\label{def:precontrast-connections}
Define the two connections $D^{(1)}$ and $D^{(-1)}$ induced by $\Psi$ through
\begin{align}
  \braket{D_{X_1}^{(1)}X_2,Y}
  &:=-\Psi(X_1X_2\mid Y),
  \label{eq:D1-definition}
  \\
  \braket{X,D_{Y_1}^{(-1)}Y_2}
  &:=-\Psi(X\mid Y_1Y_2).
  \label{eq:Dminus1-definition}
\end{align}
Set
\begin{equation}
  D^{(0)}
  :=\frac12\{D^{(1)}+D^{(-1)}\}.
  \label{eq:D0-definition}
\end{equation}
\end{definition}

Differentiating \eqref{eq:diagonal-first-identity} once more along the diagonal
gives
\begin{equation}
  X_1\braket{X_2,X_3}
  =\braket{D_{X_1}^{(1)}X_2,X_3}
   +\braket{X_2,D_{X_1}^{(-1)}X_3}.
  \label{eq:precontrast-duality}
\end{equation}
Thus $D^{(1)}$ and $D^{(-1)}$ are dual with respect to $g$. Moreover,
$D^{(-1)}$ is torsion-free. If each one-form $\Psi^\theta$ is closed, then
$D^{(1)}$ is also torsion-free; in particular, this holds if each $\Psi^\theta$
is exact. Whenever $D^{(1)}$ is torsion-free,
\begin{equation}
  D^{(0)}=\nabla.
  \label{eq:D0-equals-nabla}
\end{equation}
With the notation in \eqref{eq:D-Laplacian-definition}, we abbreviate
$\Delta^{D^{(1)}}$ and $\Delta^{D^{(-1)}}$ by $\Delta^{(1)}$ and
$\Delta^{(-1)}$, respectively.

\subsection{Loss geometry}
\label{subsec:loss-induced-geometry}

We use the following notion of a loss.
\begin{definition}[Loss]
\label{def:loss}
Let $W:\Theta\times Z\to\R$ be a continuous function. Suppose that, for every
$\theta\in\Theta$, the minimizer is unique:
\begin{equation}
  \operatorname*{argmin}_{z\in Z}W(\theta,z)=\{m(\theta)\},
  \qquad
  m(\Theta^\circ)\subset Z^\circ.
  \label{eq:loss-target}
\end{equation}
Assume that $W$ is $C^\infty$ on a neighborhood of
$\operatorname{Graph}(m)$, and that the Hessian of $W(\theta,\cdot)$ at
$m(\theta)$ is positive definite for every $\theta\in\Theta$. Then $W$ is
called a loss.
\end{definition}

The Hessian appearing in Definition~\ref{def:loss} is intrinsically defined.
We first note that $m$ is continuous. Indeed, let $\theta_k\to\theta$. By
compactness of $Z$, every subsequence of $m(\theta_k)$ has a further
subsequence converging to some $z\in Z$. Since
\[
  W(\theta_k,m(\theta_k))
  \le W(\theta_k,m(\theta)),
\]
continuity of $W$ gives $W(\theta,z)\le W(\theta,m(\theta))$ along the
convergent subsequence. By uniqueness of the minimizer, $z=m(\theta)$.
Hence $m(\theta_k)\to m(\theta)$.

For $\theta\in\Theta^\circ$, we have $m(\theta)\in Z^\circ$, and therefore
\[
  \bigl(\dd W(\theta,\cdot)\bigr)_{m(\theta)}=0.
\]
Since $m$ is continuous, $W$ is smooth near $\operatorname{Graph}(m)$, and
$\Theta^\circ$ is dense in $\Theta$, this identity extends to every
$\theta\in\Theta$. Thus $m(\theta)$ is a critical point of
$W(\theta,\cdot)$ for every $\theta$, so its Hessian at $m(\theta)$ is
intrinsically defined. We denote this Hessian by $g_W(\theta)$. Its
nondegeneracy and the implicit function theorem give
$m\in C^\infty(\Theta;Z)$; consequently $g_W$ is a bundle metric on
$m^*TZ$.

Throughout the remainder of the paper, we assume that the continuous function \(W\) fixed in Section~\ref{sec:setup} is a loss in the sense of Definition~\ref{def:loss}.

In parallel with the notation introduced for the precontrast above, we use
the following notation for derivatives of $W$ with respect to its two
arguments. For $X\in\mathfrak X(\Theta)$ and $Y\in\mathfrak X(Z)$, let
$X^{(1)}$ and $Y^{(2)}$ denote their lifts to the first and second factors
of $\Theta\times Z$, respectively. For
$X_1,\ldots,X_r\in\mathfrak X(\Theta)$ and
$Y_1,\ldots,Y_s\in\mathfrak X(Z)$, with $r,s\ge0$, set
\begin{equation}
  W(X_1\cdots X_r\mid Y_1\cdots Y_s)(\theta)
  :=
  \left.
  X_1^{(1)}\cdots X_r^{(1)}
  Y_1^{(2)}\cdots Y_s^{(2)}W
  \right|_{(\theta,m(\theta))}.
  \label{eq:loss-jet}
\end{equation}
When there is no derivative in the first or second argument, write
$W(\mid Y_1\cdots Y_s)$ or $W(X_1\cdots X_r\mid)$, respectively. For
$Y\in\mathfrak X(Z)$, write $Y\circ m\in\Gamma(m^*TZ)$ for the section
$(Y\circ m)_\theta:=Y_{m(\theta)}$. The stationarity established above and
the definition of $g_W$ give
\begin{equation}
  W(\mid Y)=0,
  \qquad
  W(\mid Y_1Y_2)
  =\braket{Y_1\circ m,Y_2\circ m}_{g_W}.
  \label{eq:loss-stationarity}
\end{equation}
For $Y_1,Y_2\in\Gamma(m^*TZ)$ we use the same notation:
\[
  W(\mid Y_1Y_2):=\braket{Y_1,Y_2}_{g_W}.
\]

For $Y\in\Gamma(m^*TZ)$ and $\theta\in\Theta$, choose a vector field
$\widetilde Y$ on $Z$ satisfying
$\widetilde Y_{m(\theta)}=Y_\theta$, and set
\[
  W(X_1\cdots X_r\mid Y)(\theta)
  :=
  W(X_1\cdots X_r\mid\widetilde Y)(\theta).
\]
This value is independent of the chosen extension. Since
$W(\mid\widetilde Y)=0$, differentiating this identity at $\theta$ in the
direction $X$ gives
\begin{equation}
\begin{aligned}
  W(X\mid\widetilde Y)(\theta)
  +W\!\left(\mid\dd m[X]\,(\widetilde Y\circ m)\right)(\theta)
  &=
  W(X\mid Y)(\theta)
  +W(\mid\dd m[X]\,Y)(\theta) \\ &  =   W(X\mid Y)(\theta)
  +\braket{\dd m[X],Y}_{g_W}(\theta) \\ &
  =0.
  \label{eq:loss-first-identity}
\end{aligned}
\end{equation}

\begin{definition}
\label{def:loss-induced-derivative}
Define
\begin{equation}
\begin{aligned}
  D^W&:\Gamma(T\Theta)
       \longrightarrow\Gamma(T^*\Theta\otimes m^*TZ)
\end{aligned}
\end{equation}
by
\begin{equation}
  \braket{D_{X_1}^WX_2,Y}_{g_W}
  :=-W(X_1X_2\mid Y)
  \label{eq:loss-derivative-definition}
\end{equation}
for $X_1,X_2\in\mathfrak X(\Theta)$ and $Y\in\Gamma(m^*TZ)$.
\end{definition}

First, for $f\in C^\infty(\Theta)$,
\[
  \braket{D_{fX_1}^WX_2,Y}_{g_W}
  =-W((fX_1)X_2\mid Y)
  =-fW(X_1X_2\mid Y)
  =f\braket{D_{X_1}^WX_2,Y}_{g_W}.
\]
Since $g_W$ is nondegenerate, $D^W$ is $C^\infty(\Theta)$-linear in $X_1$;
hence, for each $X_2\in\mathfrak X(\Theta)$,
\[
  D^W X_2\in\Gamma(T^*\Theta\otimes m^*TZ).
\]

Next, \eqref{eq:loss-first-identity} and
\eqref{eq:loss-derivative-definition} give
\begin{align*}
  \braket{D_{X_1}^W(fX_2),Y}_{g_W}
  &=-W(X_1(fX_2)\mid Y)\\
  &=-X_1(f)W(X_2\mid Y)-fW(X_1X_2\mid Y)\\
  &=\braket{X_1(f)\,\dd m[X_2]+fD_{X_1}^WX_2,Y}_{g_W}.
\end{align*}
Again by the nondegeneracy of $g_W$,
\begin{equation}
  D_{X_1}^W(fX_2)
  =X_1(f)\,\dd m[X_2]+fD_{X_1}^WX_2.
  \label{eq:loss-O-derivative-Leibniz}
\end{equation}
Therefore $D^W$ is an O-derivative with principal homomorphism
$\dd m:T\Theta\to m^*TZ$.

For $\gamma\in C^\infty(Z;\R^{\mathrm q})$, define
\[
  \dd\gamma\circ m
  \in\Gamma(m^*T^*Z\otimes\R^{\mathrm q}),
  \qquad
  (\dd\gamma\circ m)(\theta):=\dd\gamma_{m(\theta)}.
\]
Using the dual O-derivative, define
\begin{equation}
  \Delta^W\gamma
  :=\bigl(D^W(\dd\gamma\circ m)\bigr)[g^{-1}],
  \qquad
  \Delta^W:
  C^\infty(Z;\R^{\mathrm q})
  \longrightarrow
  C^\infty(\Theta;\R^{\mathrm q}).
  \label{eq:loss-Laplacian-definition}
\end{equation}

\begin{example}
\label{ex:squared-distance-loss}
Let $Z=\Theta$, let $h$ be a smooth Riemannian metric on $\Theta$, and let
$d_h$ and $\nabla^h$ denote the Riemannian distance and the Levi--Civita
connection induced by $h$, respectively.
Assume that $d_h^2$ is smooth on a neighborhood of
$\{(\theta,\theta):\theta\in\partial\Theta\}$ in
$\Theta\times\Theta$; this condition is vacuous when
$\partial\Theta=\varnothing$. Since $d_h^2$ is smooth on a neighborhood
of $(\theta,\theta)$ for every $\theta\in\Theta^\circ$,
\[
  W(\theta,z):=d_h(\theta,z)^2
\]
satisfies the conditions of Definition~\ref{def:loss}: its unique
minimizer is $m(\theta)=\theta$, and
$\operatorname{Hess}_z W(\theta,\theta)=2h_\theta$.
The decision $\widehat z_n$ is a Fr\'echet mean of $\mu_n$ with respect
to $d_h$, and
\begin{equation}
  m=\operatorname{id}_\Theta,\qquad
  g_W=2h,\qquad
  D^W=\nabla^h,\qquad
  \Delta^W=\Delta^{\nabla^h}.
  \label{eq:distance-loss-geometry}
\end{equation}
In particular, if $h=g$, then $\Delta^W=\Delta$.
\end{example}

\subsection{Assumptions}
\label{subsec:local-assumptions}
\label{subsec:localization-assumption}

Choose a chart
$x=(x^1,\ldots,x^{\mathrm p})$ around $\theta_0$ and write
$\partial_i:=\partial/\partial x^i$.

\begin{namedassumption}{A1}
\begin{align}
  \max_{1\le i\le \mathrm p}
  \left|\Psi_n[\partial_i](\theta_0)\right|
  &=O_{\Linftyminus(\theta_0)}(a_n),
  \label{eq:zeroth-jet-rate}
  \\
  \max_{1\le i,j\le \mathrm p}
  \left|
    \partial_i\{\Psi_n[\partial_j]\}(\theta_0)
    -\braket{ \partial_i,\partial_j}(\theta_0)
  \right|
  &=o_{\Linftyminus(\theta_0)}(a_n^{1/2}).
  \label{eq:first-jet-rate}
\end{align}
Moreover, for every $p<\infty$,
\begin{equation}
\begin{aligned}
  \lim_{r\downarrow0}\limsup_{n\to\infty}
  \max_{1\le i,j,k\le \mathrm p}
  \Bigg\|
    \sup_{d(\theta,\theta_0)\le r}
    \Big|
      \partial_i\partial_j\{\Psi_n[\partial_k]\}(\theta)
      -\partial_i\partial_j\{\Psi^{\theta_0}[\partial_k]\}(\theta)
    \Big|
  \Bigg\|_{L^p(P_{n,\theta_0})}
  =0.
  \label{eq:second-jet-local-rate}
\end{aligned}
\end{equation}
\end{namedassumption}

\begin{definition}
\label{def:local-scaled-jets}
Define the random smooth tensor fields
\(
  \zeta_n\in\mathfrak X(\Theta)
\)
and
\(
  \Upsilon_n\in\Gamma(T\Theta\otimes T^*\Theta)
\)
by
\begin{equation}
  \zeta_n
  :=-a_n^{-1}g^{-1}\Psi_n,
  \label{eq:zeta-definition}
\end{equation}
and
\begin{equation}
  \braket{ \Upsilon_n[X],Y}
  :=\braket{ X,Y}-(D^{(1)}\Psi_n)[X,Y],
  \qquad X,Y\in\mathfrak X(\Theta).
  \label{eq:Z-definition}
\end{equation}
\end{definition}

By \eqref{eq:zeroth-jet-rate} and \eqref{eq:first-jet-rate},
\begin{equation}
  \zeta_n(\theta_0)=O_{\Linftyminus(\theta_0)}(1),
  \qquad
  \Upsilon_n(\theta_0)=o_{\Linftyminus(\theta_0)}(a_n^{1/2}).
  \label{eq:zeta,Z-order}
\end{equation}

For $r>0$, write
\[
  B_\Theta(\theta_0,r):=\{\theta\in\Theta:d(\theta,\theta_0)<r\}.
\]
The localization condition below is a condition on the selected
sequence of invariant probability measures. To make this dependence
explicit, we write it as Assumption~\textup{[A2]}$(\mu_n)$. the choice of $(\tau,V)$ enters this condition only through $\mu_n$ for the fixed
statistical experiment, true parameter $\theta_0$, 1-form $\Psi_n$, rate $a_n$, and metric
$g$.

\begin{namedassumption}[$(\mu_n)$]{A2}
There exists $\kappa\in(0,1)$ such that, for every $L>0$,
\begin{equation}
  \E_{n,\theta_0}\!\left[
    \mu_n\!\left(
      B_\Theta(\theta_0,a_n^{\kappa})^c
    \right)
  \right]=O(a_n^L).
  \label{eq:localization-assumption}
\end{equation}
\end{namedassumption}

Let $\beta\in\Omega^1(\Theta)$ and let $D^{\mathrm B}$ be a connection,
appearing in the asymptotic conditions below.
\begin{namedassumption}{A3}
In $T_{\theta_0}^*\Theta$,
\begin{equation}
  a_n^{-2}\E_{n,\theta_0}[\Psi_n(\theta_0)]
  \longrightarrow \beta(\theta_0).
  \label{eq:beta-definition}
\end{equation}
For any smooth vector fields $X_1,X_2,X_3$ defined near $\theta_0$,
\begin{align}
  a_n^{-2}\Cov_{n,\theta_0}\!\left\{
    \Psi_n[X_1](\theta_0),
    \Psi_n[X_2](\theta_0)  \right\}
  &\longrightarrow
  \braket{X_1,X_2}(\theta_0),
  \label{eq:covariance-normalization}
  \\
  a_n^{-2}\Cov_{n,\theta_0}\!\left\{
    X_1(\Psi_n[X_2])(\theta_0),
    \Psi_n[X_3](\theta_0)
  \right\}
  &\longrightarrow
  \braket{D_{X_1}^{\mathrm B}X_2,X_3}(\theta_0).
  \label{eq:DB-definition}
\end{align}
\end{namedassumption}

The fact that the metrics in \eqref{eq:first-jet-rate} and \eqref{eq:covariance-normalization} both coincide with $g$ is ensured by the Godambe normalization established in Proposition~\ref{prop:godambe-normalization}. The Godambe normalization alone does not imply either $\beta=0$ or $D^{\mathrm B}=D^{(1)}$.

For a $(1,2)$-tensor $C$, let $\operatorname{tr}C$ denote the contraction of
its contravariant index with its first covariant index; thus, in local
coordinates, $(\operatorname{tr}C)_j=C^i{}_{ij}$. Define the smooth vector
field
\begin{equation}
  b:=g^{-1}\left\{
    \beta+\operatorname{tr}(D^{(1)}-D^{\mathrm B})
  \right\}\in\mathfrak X(\Theta).
  \label{eq:b-definition}
\end{equation}
The vector field \(b\) collects the contributions of \(\beta\) and \(D^{\mathrm B}-D^{(1)}\) that enter the second-order bias; see Proposition~\ref{prop:bias-expansion}.

\begin{namedassumption}{L}
There exists an open neighborhood $O$ of $\theta_0$ such that, for every
$n\ge1$ and $\theta_1,\theta_2\in O$, the $T^*_{\theta_1}\Theta$-valued random
variable $\Psi_n(\theta_1)$ is $P_{n,\theta_2}$-integrable. For each
$\theta\in O$, define the one-form on $O$
\begin{equation}
  \overline\Psi_n^{\theta}
  :=\E_{n,\theta}[\Psi_n].
  \label{eq:expected-one-form}
\end{equation}
Assume that each $\overline\Psi_n^{\theta}$ is smooth and satisfies
\begin{equation}
  \overline\Psi_n^{\theta}(\theta)=0,
  \qquad \theta\in O.
  \label{eq:likelihood-diagonal-centering}
\end{equation}
Assume further that the family
$\{\overline\Psi_n^{\theta}\}_{\theta\in O}$ defines a smooth map
$\overline\Psi_n:TO\times O\to\R$ by
\begin{equation}
  \overline\Psi_n(\cdot,\theta):=\overline\Psi_n^{\theta},
  \qquad \theta\in O.
  \label{eq:expected-one-form-family}
\end{equation}
Define $\overline\Psi_n(\,\cdot\mid\cdot\,)$ using the same convention as in
\eqref{eq:precontrast-jet}. For all smooth vector fields $X_1,X_2,X_3$ on $O$,
\begin{align}
  \overline\Psi_n(X_1\mid X_2)(\theta_0)
  &=-a_n^{-2}\Cov_{n,\theta_0}\!\left\{
    \Psi_n[X_1](\theta_0),\Psi_n[X_2](\theta_0)
  \right\},
  \label{eq:likelihood-score-covariance}
  \\
  \overline\Psi_n(X_1X_2\mid X_3)(\theta_0)
  &=-a_n^{-2}\Cov_{n,\theta_0}\!\left\{
    X_1(\Psi_n[X_2])(\theta_0),\Psi_n[X_3](\theta_0)
  \right\}.
  \label{eq:likelihood-mixed-covariance}
\end{align}
Finally, for every smooth vector field $X_1,X_2,X_3$ on $O$,
\begin{equation}
  \overline\Psi_n(X_1X_2\mid X_3)(\theta_0)
  \longrightarrow
  \Psi(X_1X_2\mid X_3)(\theta_0).
  \label{eq:likelihood-mixed-jet}
\end{equation}
\end{namedassumption}

\begin{proposition}[Likelihood simplification]
\label{prop:likelihood-bartlett-compatibility}
Under Assumptions~\textup{[L]} and \textup{[A1]}, Assumption~\textup{[A3]}
holds with $b=0$. More precisely,
$\beta=0$ and $D^{\mathrm B}=D^{(1)}$.
\end{proposition}

\subsection{Euclidean setting}
\label{subsec:euclidean-setting}

We invoke the following Euclidean setting only when stated explicitly.

\begin{namedsetting}{E}
Suppose that $\Theta\subset\R^{\mathrm p}$ is a compact domain with smooth
boundary. By abuse of notation, we use
$\theta=(\theta^1,\ldots,\theta^{\mathrm p})$ both for a point of
$\Theta$ and for the standard coordinate map
$\operatorname{id}_\Theta:\Theta\to\R^{\mathrm p}$.

Whenever Setting~\textup{[E]} is in force, we use the canonical
identifications
\[
  T_\theta\Theta\simeq\R^{\mathrm p},
  \qquad
  T_\theta^*\Theta\simeq(\R^{\mathrm p})^*,
\]
and identify $\zeta_n$, $\Upsilon_n$, $g^{-1}$, and $V$ with their coordinate
representations. For a connection $D$ on $T\Theta$, let $\Gamma^D$ denote
its Christoffel coefficients in the standard coordinates. For a
$(2,0)$-tensor $S$, write
\[
  \Gamma^D[S]
  :=\bigl(\Gamma^{D,k}_{ij}S^{ij}\bigr)_{k=1}^{\mathrm p}.
\]
We abbreviate
\[
  \Gamma^{(1)}:=\Gamma^{D^{(1)}},
  \qquad
  \Gamma^{(-1)}:=\Gamma^{D^{(-1)}}.
\]
\end{namedsetting}

\section{Second-Order Stochastic Expansion}
\label{sec:main-expansion}

Define the random differential operator
\begin{equation}
  \mathcal M_n:
  C^\infty(\Theta;\R^{\mathrm q})
  \longrightarrow C^\infty(\Theta;\R^{\mathrm q})
  \label{eq:expansion-operator-domain}
\end{equation}
by
\begin{equation}
\begin{aligned}
  \mathcal M_n f
  :={}& f+a_n\,\dd f[\zeta_n+\Upsilon_n[\zeta_n]]
  \\
  &+a_n^2\Bigg\{
    \dd f[V]
    +\left(\tau\Delta+\frac{1-\tau}{2}\Delta^{(-1)}\right)f
    +\frac12(D^{(-1)}\dd f)
       [\zeta_n^{\otimes2}-g^{-1}]
  \Bigg\}.
  \label{eq:expansion-operator}
\end{aligned}
\end{equation}

\begin{theorem}
\label{thm:main-expansion}
Under Assumption~\textup{[A1]} and
Assumption~\textup{[A2]}$(\mu_n)$, for every fixed
$f\in C^\infty(\Theta;\R^{\mathrm q})$ and
$\gamma\in C^\infty(Z;\R^{\mathrm q})$,
\begin{align}
  \mu_n(f)
  &=(\mathcal M_nf)(\theta_0)
    +o_{\Linftyminus(\theta_0)}(a_n^2),
  \label{eq:main-expansion}\\
  \gamma(\widehat z_n)
  &=\bigl(\mathcal M_n(\gamma\circ m)\bigr)(\theta_0)
    -\frac{\tau a_n^2}{2}(\Delta^W\gamma)(\theta_0)
    +o_{\Linftyminus(\theta_0)}(a_n^2).
  \label{eq:loss-explicit-expansion}
\end{align}
\end{theorem}

The additional moment conditions in \textup{[A3]} identify the expectations
of the random terms in this expansion. Define the deterministic differential
operator
$\calB:C^\infty(\Theta;\R^{\mathrm q})\to C^\infty(\Theta;\R^{\mathrm q})$ by
\begin{equation}
  \calB f
  :=\dd f[V-b]
    +\tau\Delta f
    +\frac{1-\tau}{2}\Delta^{(-1)}f.
  \label{eq:general-bias-formula}
\end{equation}
The operator $\calB$ is independent of the loss $W$ and its target map $m$.

\begin{proposition}
\label{prop:bias-expansion}
Suppose that Assumptions~\textup{[A1]} and~\textup{[A3]} hold. Then
\begin{align}
  \E_{n,\theta_0}\!\left[
    (\zeta_n+\Upsilon_n[\zeta_n])(\theta_0)
  \right]
  &=-a_n b(\theta_0)+o(a_n),
  \label{eq:scaled-jet-first-moment}\\
  \E_{n,\theta_0}\!\left[
    \zeta_n(\theta_0)^{\otimes2}
  \right]
  &=g^{-1}(\theta_0)+o(1).
  \label{eq:scaled-jet-second-moment}
\end{align}
Consequently, if Assumption~\textup{[A2]}$(\mu_n)$ also holds,
then, for every fixed
$f\in C^\infty(\Theta;\R^{\mathrm q})$ and
$\gamma\in C^\infty(Z;\R^{\mathrm q})$,
\begin{align}
  \E_{n,\theta_0}[\mu_n(f)]-f(\theta_0)
  &=a_n^2(\calB f)(\theta_0)+o(a_n^2),
  \label{eq:bias-expansion}\\
  \E_{n,\theta_0}[\gamma(\widehat z_n)]-\gamma(m(\theta_0))
  &=a_n^2\left\{
      \calB(\gamma\circ m)-\frac{\tau}{2}\Delta^W\gamma
    \right\}(\theta_0)+o(a_n^2).
  \label{eq:loss-bias-expansion}
\end{align}
\end{proposition}

We next compare two choices $(\tau_j,V_j,W_j)$, $j=1,2$, while
keeping the statistical input $\Psi_n$, the rate $a_n$, and the metric $g$
fixed, and hence the same precontrast geometry. Let
$W_j:\Theta\times Z\to\R$ be a loss, let $\tau_j\ge0$, and let
$V_j\in\mathfrak X(\Theta)$. Assume
that $W_1$ and $W_2$ have the same target map $m:\Theta\to Z$. Let
$L_n^{(j)}$ be the operator \eqref{eq:full-generator} determined by
$(\tau_j,V_j)$, let $\mu_n^{(j)}$ be the selected invariant probability
measure, and let $\widehat z_n^{(j)}$ be the decision rule
\eqref{eq:decision-rule} determined by $(W_j,\mu_n^{(j)})$. We write
$\mathcal M_n^{(j)}$ and $\calB^{(j)}$ for the corresponding operators
defined above.
The difference is independent of $b$:
\begin{equation}
\begin{aligned}
  (\calB^{(1)}-\calB^{(2)})f
  &=\dd f[V_1-V_2]
    +(\tau_1-\tau_2)
      \left(\Delta-\frac12\Delta^{(-1)}\right)f,\\
  (\mathcal M_n^{(1)}-\mathcal M_n^{(2)})f
  &=a_n^2(\calB^{(1)}-\calB^{(2)})f.
  \label{eq:estimator-comparison-operator}
\end{aligned}
\end{equation}
In particular, this operator difference does not require
Assumption~\textup{[A3]}. If $D^{(1)}$ is torsion-free, then
$\Delta-\tfrac12\Delta^{(-1)}=\tfrac12\Delta^{(1)}$.

\begin{corollary}
\label{cor:loss-temperature-drift-comparison}
 Assume Assumption~\textup{[A1]} together with
Assumptions~\textup{[A2]}$(\mu_n^{(1)})$ and
\textup{[A2]}$(\mu_n^{(2)})$. For every fixed $f\in C^\infty(\Theta;\R^{\mathrm q})$ and
$\gamma\in C^\infty(Z;\R^{\mathrm q})$,
\begin{align}
  \mu_n^{(1)}(f)-\mu_n^{(2)}(f)
  &=a_n^2\{(\calB^{(1)}-\calB^{(2)})f\}(\theta_0)
    +o_{\Linftyminus(\theta_0)}(a_n^2),
  \label{eq:integral-design-comparison}\\
  \gamma(\widehat z_n^{(1)})-\gamma(\widehat z_n^{(2)})
  &=a_n^2\Bigl\{
      (\calB^{(1)}-\calB^{(2)})(\gamma\circ m)
-\frac12\bigl(
      \tau_1\Delta^{W_1}\gamma-\tau_2\Delta^{W_2}\gamma
    \bigr)\Bigr\}(\theta_0)
 +o_{\Linftyminus(\theta_0)}(a_n^2).
  \label{eq:general-loss-design-comparison}
\end{align}
\end{corollary}

The leading terms in both expansions are deterministic. Since the
remainders are $o_{L^1(P_{n,\theta_0})}(a_n^2)$, taking expectations gives
the corresponding bias comparisons without Assumption~\textup{[A3]}.

\subsection{Coordinate formulas}
\label{subsec:coordinate-formulas}

\begin{corollary}
\label{cor:coordinate-expansion}
Suppose Setting~\textup{[E]} holds. Under Assumption~\textup{[A1]} and
Assumption~\textup{[A2]}$(\mu_n)$,
\begin{equation}
\begin{aligned}
  \mu_n(\theta)
  ={}&\theta_0+a_n\{\zeta_n+\Upsilon_n[\zeta_n]\}(\theta_0)\\
  &+a_n^2\Bigg\{
      V-\left(\tau\Gamma^\nabla
             +\frac{1-\tau}{2}\Gamma^{(-1)}\right)[g^{-1}]
      -\frac12\Gamma^{(-1)}[\zeta_n^{\otimes2}-g^{-1}]
    \Bigg\}(\theta_0)+o_{\Linftyminus(\theta_0)}(a_n^2).
  \label{eq:euclidean-parameter-expansion}
\end{aligned}
\end{equation}
If Assumption~\textup{[A1]} and
Assumptions~\textup{[A2]}$(\mu_n^{(1)})$ and
\textup{[A2]}$(\mu_n^{(2)})$ hold, then
\begin{equation}
\begin{aligned}
  &\mu_n^{(1)}(\theta)-\mu_n^{(2)}(\theta)\\
  &\quad=a_n^2\left\{
      V_1-V_2-(\tau_1-\tau_2)
        \left(\Gamma^\nabla-\frac12\Gamma^{(-1)}\right)[g^{-1}]
    \right\}(\theta_0)
    +o_{\Linftyminus(\theta_0)}(a_n^2).
  \label{eq:estimator-comparison-coordinate}
\end{aligned}
\end{equation}
If $D^{(1)}$ is torsion-free, the connection coefficient in parentheses is
$\tfrac12\Gamma^{(1)}$.
If Assumption~\textup{[A3]} also holds, then
\begin{equation}
\begin{aligned}
  \E_{n,\theta_0}[\mu_n(\theta)]-\theta_0
  =a_n^2\Bigg[
      V-b-\left\{
        \tau\Gamma^\nabla
        +\frac{1-\tau}{2}\Gamma^{(-1)}
      \right\}[g^{-1}]
    \Bigg](\theta_0)
    +o(a_n^2).
  \label{eq:coordinate-bias-expansion}
\end{aligned}
\end{equation}
In particular, if
\begin{equation}
  V
  =
  b+\left\{
    \tau\Gamma^\nabla
    +\frac{1-\tau}{2}\Gamma^{(-1)}
  \right\}[g^{-1}],
  \label{eq:coordinate-second-order-unbiased-drift}
\end{equation}
then
\begin{equation}
  \E_{n,\theta_0}[\mu_n(\theta)]-\theta_0=o(a_n^2).
  \label{eq:coordinate-second-order-unbiased}
\end{equation}
\end{corollary}

\subsection{Relation to existing results and examples}
\label{subsec:relations-examples}

\subsubsection{First-order asymptotics}

At first order, Theorem~\ref{thm:main-expansion} together with
\eqref{eq:zeta,Z-order} gives
\begin{align}
  a_n^{-1}\{\mu_n(f)-f(\theta_0)\}
  &=\dd f[\zeta_n](\theta_0)
      +o_{\Linftyminus(\theta_0)}(1),
  \label{eq:integral-first-order-expansion}\\
  a_n^{-1}\{\gamma(\widehat z_n)-\gamma(m(\theta_0))\}
  &=\dd(\gamma\circ m)[\zeta_n](\theta_0)
      +o_{\Linftyminus(\theta_0)}(1).
  \label{eq:loss-first-order-expansion}
\end{align}
The leading terms have the familiar first-order form from likelihood,
quasi-likelihood, and Bayes-type asymptotics. In likelihood and
quasi-likelihood settings, Ibragimov--Has'minskii theory
\cite{IbragimovHasminskii1981} and QLA
\cite{Yoshida2011,Yoshida2025Simplified} develop such first-order
asymptotics together with moment convergence under their own localization
and large-deviation conditions. Ogihara~\cite{Ogihara2019} treats
general-loss Bayes-type estimators, including first-order equivalence and
moment convergence, while Abraham and Cadre~\cite{AbrahamCadre2004} study
asymptotic Bayes actions on a decision space distinct from the parameter
space.

Here the $L^{\infty-}$ control is obtained under
Assumption~\textup{[A2]}$(\mu_n)$, which is imposed directly. Theorem~\ref{thm:main-expansion}
retains the explicit second-order terms beyond the first-order structure,
and Proposition~\ref{prop:bias-expansion} identifies their expectations
under the additional moment conditions in \textup{[A3]}.

\subsubsection{Bias reduction}

Consider the likelihood case
\eqref{eq:intro-likelihood-specialization} under Setting~\textup{[E]}.
For fixed $\tau\ge0$ and $V\in\mathfrak X(\Theta)$, assume
Assumptions~\textup{[A1]} and \textup{[L]} and
Assumption~\textup{[A2]}$(\mu_n)$.
Proposition~\ref{prop:likelihood-bartlett-compatibility} then gives
Assumption~\textup{[A3]} with $b=0$. Hence
Corollary~\ref{cor:coordinate-expansion} shows that the leading coordinate
bias vanishes when
\[
  V
  =
  \left\{
    \tau\Gamma^\nabla
    +\frac{1-\tau}{2}\Gamma^{(-1)}
  \right\}[g^{-1}].
\]

At zero temperature, this becomes
\begin{equation}
  V
  =
  \frac12\Gamma^{(-1)}[g^{-1}].
  \label{eq:zero-temperature-bias-correcting-field}
\end{equation}
In regular i.i.d.\ likelihood models, this gives the
information-geometric form of Firth's mean-bias-reducing score
adjustment~\cite{Firth1993}. For
$f\in C^\infty(\Theta;\R^{\mathrm q})$, the zero-temperature
specialization of \eqref{eq:bias-expansion} contains the bias-reduction
formulas studied by Hirose and Mano~\cite{HiroseMano2026}. More generally, the zero-temperature
specialization of \eqref{eq:coordinate-second-order-unbiased-drift} is
consistent with the bias-reducing adjustment framework for estimation
equations of Kosmidis and Lunardon~\cite{KosmidisLunardon2024}.

At unit temperature, the corresponding choice is
\begin{equation}
  V
  =
  \Gamma^\nabla[g^{-1}].
  \label{eq:unit-temperature-coordinate-bias-correcting-field}
\end{equation}
To relate this choice to prior-based bias reduction, fix a smooth positive
reference measure $\nu$ and write
\[
  \pi_J:=\frac{\dd\vol}{\dd\nu}.
\]
Let $\Pi_{\mathrm{BR}}$ be a prior with smooth positive density
$\pi_{\mathrm{BR}}:=\dd\Pi_{\mathrm{BR}}/\dd\nu$. Note that
$\pi_{\mathrm{BR}}/\pi_J=\dd\Pi_{\mathrm{BR}}/\dd\vol$.
Under the assumptions above at every interior true parameter, the posterior
mean with respect to $\Pi_{\mathrm{BR}}$ is second-order unbiased if and only if
\begin{equation}
  g^{-1}\dd\log\frac{\pi_{\mathrm{BR}}}{\pi_J}
  =
  \Gamma^\nabla[g^{-1}].
  \label{eq:posterior-mean-bias-reducing-prior}
\end{equation}
Equation~\eqref{eq:posterior-mean-bias-reducing-prior} is the
information-geometric form of the asymptotically unbiased-prior condition
originating with Hartigan~\cite{Hartigan1965} and subsequently developed by
Sakai, Matsuda and Kubokawa~\cite{SakaiMatsudaKubokawa2025} and by Miyata and
Yanagimoto~\cite{MiyataYanagimoto2026}. 

\subsubsection{MLE, posterior mean and MAP estimator}
Continue in the likelihood case under Setting~\textup{[E]}, and suppose
$D^{(1)}$ is torsion-free. Retain the reference measure $\nu$ and the
notation \(  \pi_J = \dd \vol/\dd \nu\) from the preceding discussion.
We represent the three estimators as follows.

For the MLE, take
\[
  (\tau_{\mathrm{MLE}},V_{\mathrm{MLE}})=(0,0),
\]
and let $\mu_n^{\mathrm{MLE}}$ be the selected invariant Dirac measure at an
interior MLE $\widehat\theta_n^{\mathrm{MLE}}$. Then
\[
  \widehat\theta_n^{\mathrm{MLE}}
  =\mu_n^{\mathrm{MLE}}(\theta).
\]

For the posterior mean, let $\Pi_{\mathrm{PM}}$ be a prior with smooth
positive density
\[
  \pi_{\mathrm{PM}}:=\frac{\dd\Pi_{\mathrm{PM}}}{\dd\nu}.
\]
Since
$\pi_{\mathrm{PM}}/\pi_J=\dd\Pi_{\mathrm{PM}}/\dd\vol$, take
\[
  (\tau_{\mathrm{PM}},V_{\mathrm{PM}})
  =
  \left(
    1,
    g^{-1}\dd\log\frac{\pi_{\mathrm{PM}}}{\pi_J}
  \right).
\]
The corresponding invariant probability measure is the posterior
\[
  \mu_n^{\mathrm{PM}}(\dd\theta)
  \propto
  e^{\ell_n(\theta)}
  \pi_{\mathrm{PM}}(\theta)\,\nu(\dd\theta),
\]
and
\[
  \widehat\theta_n^{\mathrm{PM}}
  :=\mu_n^{\mathrm{PM}}(\theta).
\]
The posterior distribution $\mu_n^{\mathrm{PM}}$, and hence
$\widehat\theta_n^{\mathrm{PM}}$, is independent of the choice of $\nu$.

For the MAP estimator, let $\Pi_{\mathrm{MAP}}$ have smooth positive density
\[
  \pi_{\mathrm{MAP}}
  :=\frac{\dd\Pi_{\mathrm{MAP}}}{\dd\nu},
\]
and let $\widehat\theta_n^{\mathrm{MAP}}$ be an interior maximizer satisfying
\begin{equation}
  \widehat\theta_n^{\mathrm{MAP}}
  \in\operatorname*{argmax}_{\theta\in\Theta}
       \{\ell_n(\theta)+\log\pi_{\mathrm{MAP}}(\theta)\}.
  \label{eq:map-reference-measure}
\end{equation}
Take
\[
  (\tau_{\mathrm{MAP}},V_{\mathrm{MAP}})
  =
  \left(
    0,
    g^{-1}\dd\log\pi_{\mathrm{MAP}}
  \right),
\]
and let
\[
  \mu_n^{\mathrm{MAP}}
  :=
  \delta_{\widehat\theta_n^{\mathrm{MAP}}}.
\]
Then
\[
  \widehat\theta_n^{\mathrm{MAP}}
  =\mu_n^{\mathrm{MAP}}(\theta).
\]
Unlike the posterior mean, the MAP estimator generally depends on the
reference measure $\nu$.

Assume Assumption~\textup{[A1]} and
Assumptions~\textup{[A2]}$(\mu_n^{\mathrm{MLE}})$,
\textup{[A2]}$(\mu_n^{\mathrm{PM}})$, and
\textup{[A2]}$(\mu_n^{\mathrm{MAP}})$.
The comparisons below follow directly from
\eqref{eq:estimator-comparison-coordinate}; Assumptions~\textup{[L]}
and~\textup{[A3]} are not needed. Since $D^{(1)}$ is torsion-free,
\[
  \Gamma^\nabla-\frac12\Gamma^{(-1)}
  =
  \frac12\Gamma^{(1)}.
\]
For the posterior mean and the MLE,
\begin{equation}
\begin{aligned}
  \widehat\theta_n^{\mathrm{PM}}
  -\widehat\theta_n^{\mathrm{MLE}}
  =a_n^2\left\{
       g^{-1}\dd\log\frac{\pi_{\mathrm{PM}}}{\pi_J}
       -\frac12\Gamma^{(1)}[g^{-1}]
    \right\}(\theta_0)
       +o_{\Linftyminus(\theta_0)}(a_n^2).
  \label{eq:posterior-mle-coordinate-comparison}
\end{aligned}
\end{equation}
At a fixed $\theta_0$, the posterior mean and the MLE agree up to
$o_{\Linftyminus(\theta_0)}(a_n^2)$ if and only if the coefficient on the
right-hand side vanishes at $\theta_0$. If the assumptions hold at every
interior true parameter, they agree to this order at every interior true
parameter if and only if
\begin{equation}
  g^{-1}\dd\log\frac{\pi_{\mathrm{PM}}}{\pi_J}
  =
  \frac12\Gamma^{(1)}[g^{-1}].
  \label{eq:moment-matching-prior}
\end{equation}
Likewise,
\begin{equation}
\begin{aligned}
  \widehat\theta_n^{\mathrm{PM}}
  -\widehat\theta_n^{\mathrm{MAP}}
  =a_n^2\left\{
       g^{-1}\dd\log\frac{\pi_{\mathrm{PM}}}
                            {\pi_{\mathrm{MAP}}\pi_J}
       -\frac12\Gamma^{(1)}[g^{-1}]
    \right\}(\theta_0)
       +o_{\Linftyminus(\theta_0)}(a_n^2),
  \label{eq:posterior-map-coordinate-comparison}
\end{aligned}
\end{equation}
whereas
\begin{equation}
  \widehat\theta_n^{\mathrm{MAP}}
  -\widehat\theta_n^{\mathrm{MLE}}
  =
  a_n^2\bigl(g^{-1}\dd\log\pi_{\mathrm{MAP}}\bigr)(\theta_0)
  +o_{\Linftyminus(\theta_0)}(a_n^2).
  \label{eq:map-mle-coordinate-comparison}
\end{equation}

In regular i.i.d.\ likelihood models,
equation~\eqref{eq:moment-matching-prior} is Tanaka's
information-geometric representation~\cite{Tanaka2023} of the first-moment
matching condition of Ghosh and Liu~\cite{GhoshLiu2011}.
Furthermore,  \eqref{eq:posterior-map-coordinate-comparison} recovers the
matching-prior-pair condition of Okudo and Yano~\cite{OkudoYano2026}.

Finally,
\[
  \Gamma^\nabla[g^{-1}]
  =
  \frac12\Gamma^{(1)}[g^{-1}]
  +
  \frac12\Gamma^{(-1)}[g^{-1}].
\]
Thus the posterior-mean bias-reduction condition
\eqref{eq:posterior-mean-bias-reducing-prior} is the sum of the
moment-matching condition \eqref{eq:moment-matching-prior} and the
zero-temperature condition
\eqref{eq:zero-temperature-bias-correcting-field}; explicitly,
\[
  g^{-1}\dd\log\frac{\pi_{\mathrm{BR}}}{\pi_J}
  =
  \underbrace{\frac12\Gamma^{(1)}[g^{-1}]}_{\text{moment matching}}
  +
  \underbrace{\frac12\Gamma^{(-1)}[g^{-1}]}_{\text{zero-temperature bias reduction}}.
\]
This recovers the decomposition described by Sakai, Matsuda and
Kubokawa~\cite{SakaiMatsudaKubokawa2025}; the same relationship is also
discussed by Miyata and Yanagimoto~\cite{MiyataYanagimoto2026}.
\subsubsection{Loss-dependent decisions}

Fix $(\tau,V)$, and let $\mu_n$ be the corresponding invariant
probability measure. Let $W_1,W_2$ be two losses with the same target map
$m$. For $j=1,2$, define
\[
  \widehat z_n^{(j)}\in\operatorname*{argmin}_{z\in Z}
  \int_\Theta W_j(\theta,z)\,\mu_n(\dd\theta),\qquad j=1,2.
\]
Under Assumption~\textup{[A1]} and
Assumption~\textup{[A2]}$(\mu_n)$,
Corollary~\ref{cor:loss-temperature-drift-comparison} yields
\begin{equation}
\begin{aligned}
  \gamma(\widehat z_n^{(1)})-\gamma(\widehat z_n^{(2)})
  &=-\frac{\tau a_n^2}{2}
      (\Delta^{W_1}-\Delta^{W_2})\gamma(\theta_0)
    +o_{\Linftyminus(\theta_0)}(a_n^2).
  \label{eq:common-target-loss-comparison}
\end{aligned}
\end{equation}
Thus the two decisions agree at first order, while their second-order
difference is determined by the loss-induced operators.

\begin{example}
\label{ex:linex-loss}
Let $f\in C^\infty(\Theta;\R)$, let $Z$ be a compact interval containing
$f(\Theta)$ in its interior, and fix $c\ne0$. Consider the LINEX loss,
up to a positive multiplicative constant,
\begin{equation}
  W(\theta,z)
  :=\frac{e^{c(z-f(\theta))}-c(z-f(\theta))-1}{c^2}.
  \label{eq:linex-loss}
\end{equation}
Bayes estimation under LINEX loss is studied by
Zellner~\cite{Zellner1986}. Here $m=f$, and
$W(\mid\partial_z\partial_z)=1$, so \eqref{eq:linex-loss} is a loss in the sense of Definition~\ref{def:loss}. Direct differentiation gives
\[
  \mathrm D Q_n(z)
  =\frac{e^{cz}\mu_n(e^{-cf})-1}{c},
  \qquad
  \mathrm D^2Q_n(z)=e^{cz}\mu_n(e^{-cf})>0.
\]
Here $\mathrm D$ denotes the Fr\'echet derivative; since $Z\subset\R$,
it is the ordinary derivative with respect to $z$.
Consequently, the decision is uniquely given by
\begin{equation}
  \widehat z_n=-\frac1c\log\mu_n(e^{-cf}).
  \label{eq:linex-decision}
\end{equation}
This value lies between $\min_\Theta f$ and $\max_\Theta f$, hence in
$Z^\circ$. The definition of $D^W$ gives, for local vector fields $X,Y$,
\[
  D_X^WY=\{XYf-c(Xf)(Yf)\}\partial_z,
\]
and therefore
\begin{equation}
  \Delta^W\operatorname{id}_Z=c\braket{\dd f,\dd f}.
  \label{eq:linex-loss-laplacian}
\end{equation}
The squared loss $\tfrac12(z-f(\theta))^2$ has the same target map $f$ and
has decision $\mu_n(f)$. Under Assumption~\textup{[A1]} and
Assumption~\textup{[A2]}$(\mu_n)$,
Theorem~\ref{thm:main-expansion} and
\eqref{eq:linex-loss-laplacian} give
\begin{equation}
  \widehat z_n-\mu_n(f)
  =-\frac{\tau c\,a_n^2}{2}\braket{\dd f,\dd f}(\theta_0)
    +o_{\Linftyminus(\theta_0)}(a_n^2).
  \label{eq:linex-comparison}
\end{equation}
If Assumption~\textup{[A3]} also holds, then
Proposition~\ref{prop:bias-expansion} gives
\begin{equation}
  \E_{n,\theta_0}[\widehat z_n]-f(\theta_0)
  =a_n^2\left\{
      \calB f-\frac{\tau c}{2}\braket{\dd f,\dd f}
    \right\}(\theta_0)+o(a_n^2).
  \label{eq:linex-bias}
\end{equation}
Thus the asymmetric loss changes the decision at second order while
preserving its target and first-order term.
\end{example}

\section{Universal Bias Reduction}
\label{sec:bias}

Throughout this section, fix a loss $W$, a temperature $\tau>0$, and
a smooth vector field $V$.

For a connection $D$ on $T\Theta$, define
\[
  (\dd m\circ D)_X Y:=\dd m[D_XY],
  \qquad X,Y\in\mathfrak X(\Theta).
\]
$\dd m\circ D$ is an O-derivative with principal homomorphism $\dd m$. Hence
$\dd m\circ D-D^W$ is a section of $m^*TZ\otimes(T^*\Theta)^{\otimes2}$.

Define the smooth section $V^*\in\Gamma(m^*TZ)$ by
\begin{equation}
\begin{aligned}
  V^*
  :=\dd m[b]+\Bigg\{
       \tau\bigl(\dd m\circ\nabla-D^W\bigr)
       +\frac{1-\tau}{2}
          \bigl(\dd m\circ D^{(-1)}-D^W\bigr)
     \Bigg\}[g^{-1}].
  \label{eq:loss-calibration-section}
\end{aligned}
\end{equation}
The section $V^*$ depends on the fixed $(\tau,W)$ but not on $V$.
When $m=\operatorname{id}_\Theta$, it is a vector field on $\Theta$.

Define $\mathbb B_{n,\theta},\widehat{\mathbb B}_n:
C^\infty(Z;\R^{\mathrm q})\to\R^{\mathrm q}$ by
\begin{equation}
\begin{aligned}
  \mathbb B_{n,\theta}(\gamma)
  &:=a_n^{-2}\{
      \E_{n,\theta}[\gamma(\widehat z_n)]-\gamma(m(\theta))
    \},\\
  \widehat{\mathbb B}_n(\gamma)
  &:=a_n^{-2}\tau^{-1}\{
      \mu_n(\gamma\circ m)-\gamma(\widehat z_n)
    \}.
  \label{eq:bias-functionals}
\end{aligned}
\end{equation}
For fixed $\theta$, the first functional is deterministic and the second
is data-dependent. Set
\begin{equation}
\begin{aligned}
  \widetilde\gamma_n
  &:=\gamma(\widehat z_n)-a_n^2\widehat{\mathbb B}_n(\gamma)
   =\frac{1+\tau}{\tau}\gamma(\widehat z_n)
      -\frac1\tau\mu_n(\gamma\circ m).
  \label{eq:loss-bias-correction}
\end{aligned}
\end{equation}

\begin{corollary}
\label{cor:bias-comparison-correction}
Under Assumptions~\textup{[A1]} and~\textup{[A3]} and
Assumption~\textup{[A2]}$(\mu_n)$, for every fixed
$\gamma\in C^\infty(Z;\R^{\mathrm q})$,
\begin{equation}
\begin{aligned}
  (\mathbb B_{n,\theta_0}-\widehat{\mathbb B}_n)(\gamma)
  &=(\dd\gamma\circ m)[\dd m[V]-V^*](\theta_0)
    +o_{\Linftyminus(\theta_0)}(1).
  \label{eq:normalized-bias-comparison}
\end{aligned}
\end{equation}
Consequently, if
\begin{equation}
  \dd m[V]=V^*,
  \label{eq:loss-calibrating-drift}
\end{equation}
then
\begin{equation}
  \E_{n,\theta_0}[\widetilde\gamma_n]-\gamma(m(\theta_0))
  =o(a_n^2).
  \label{eq:loss-corrected-bias-expansion}
\end{equation}
\end{corollary}

\begin{remark}
For comparison with bootstrap bias correction, suppose $Z=\Theta$ and
$m=\operatorname{id}_\Theta$. Conditional on the observed data, let
\[
  X_n^*\sim P_{n,\widehat z_n},
\]
and let $\widehat z_n^*$ be obtained from $X_n^*$ by applying the same
estimation procedure as for $\widehat z_n$. Write $\E^*$ for conditional
expectation given the observed data with respect to
$P_{n,\widehat z_n}$. The usual bootstrap bias correction for
$\gamma(\widehat z_n)$ has the form
\[
  2\gamma(\widehat z_n)-\E^*[\gamma(\widehat z_n^*)].
\]
At $\tau=1$, our correction is
\[
  2\gamma(\widehat z_n)-\mu_n(\gamma).
\]
Thus, at the level of the bias-correction formula, integration with respect
to $\mu_n$ plays a role analogous to averaging over the bootstrap
distribution of $\widehat z_n^*$. Unlike the bootstrap, however, the
correction is obtained from the same invariant probability measure used to
construct $\widehat z_n$, without generating bootstrap samples and
recomputing the estimator.
\end{remark}

\begin{corollary}
\label{thm:canonical-loss-calibrations}
Let $\tau=1$ and $Z=\Theta$.
\begin{enumerate}[label=\textup{(\roman*)}]
\item Suppose $W(\theta,z)=d(\theta,z)^2$, where $d^2$ satisfies the
boundary smoothness condition in Example~\ref{ex:squared-distance-loss}.
Then $V=b$ satisfies \eqref{eq:loss-calibrating-drift}.
\item Under Setting~\textup{[E]}, suppose
$W(\theta,z)=|\theta-z|^2$. Then
$V=b+\Gamma^\nabla[g^{-1}]$ satisfies
\eqref{eq:loss-calibrating-drift}.
\end{enumerate}
\end{corollary}

If, in addition, Assumptions~\textup{[A1]} and~\textup{[A3]} and
Assumption~\textup{[A2]}$(\mu_n)$ hold in either case, then
Corollary~\ref{cor:bias-comparison-correction} applies to the corrected
estimator $\widetilde\gamma_n$ in \eqref{eq:loss-bias-correction}, and
\eqref{eq:loss-corrected-bias-expansion} holds with
$m=\operatorname{id}_\Theta$.

In \textup{(i)}, $\widehat z_n$ is a Fr\'echet mean of $\mu_n$ with respect
to $d$. If, moreover, the likelihood setting
\eqref{eq:intro-likelihood-specialization} and Assumption~\textup{[L]} hold,
then Proposition~\ref{prop:likelihood-bartlett-compatibility} gives $b=0$,
and hence the choice in \textup{(i)} is $V=0$. The
corresponding invariant probability measure is
the Jeffreys posterior
\begin{equation}
  \mu_n(\dd\theta)\propto e^{\ell_n(\theta)}\,\vol(\dd\theta).
  \label{eq:jeffreys-posterior}
\end{equation}
Thus, whenever the assumptions of
Corollary~\ref{cor:bias-comparison-correction} hold,
\eqref{eq:loss-corrected-bias-expansion} holds with the Jeffreys posterior
and its Fisher--Rao Fr\'echet mean.

In \textup{(ii)}, the coordinate representation of the process is
\begin{equation}
\begin{aligned}
  &\dd\theta_t^i
  =-\frac12\left\{
      a_n^{-2}g^{ij}(\theta_t)\Psi_{n,j}(\theta_t)-b^i(\theta_t)
    \right\}\dd t + \sigma^i{}_{\alpha}(\theta_t)\,\dd W_t^\alpha
   -\mathbf n^i(\theta_t)\,\dd K_t,
\end{aligned}
  \label{eq:natural-gradient-diffusion}
\end{equation}
where $\sigma\sigma^{\mathsf T}=g^{-1}$. Under the likelihood setting \eqref{eq:intro-likelihood-specialization} and Assumption~\textup{[L]}, $b=0$ and this reduces to \eqref{eq:intro-natural-gradient-diffusion}.

\section{Simulation}
\label{sec:simulation}

We numerically examine whether $\mu_n(\theta)$, where $\mu_n$ is an
invariant probability measure of the reflected process
\eqref{eq:intro-natural-gradient-diffusion}, has the $o(n^{-1})$ bias
predicted by \eqref{eq:intro-invariant-mean-bias} in the gamma shape--scale
model. Verification of the assumptions required for this asymptotic result
is not pursued here.

\subsection{Gamma model}
\label{subsec:gamma-model}

\subsubsection{Model and the obstruction to a bias-reducing prior}

Let
\begin{equation}
  \Theta=Z=\Theta_{\mathrm{num}}
  :=[0.05,30]\times[0.02,20],
  \qquad \theta=(a,b),
  \label{eq:gamma-numerical-domain}
\end{equation}
and let $\theta_0\in\Theta^\circ$. The rectangular domain
$\Theta_{\mathrm{num}}$ has corners, whereas
Sections~\ref{sec:setup}--\ref{sec:bias} assume a smooth boundary. In the
simulations reported below, the process did not hit the boundary, so this
distinction does not affect the reported numerical results.

For each $n$, consider
\[
  \mathcal X_n=(0,\infty)^n,\qquad
  \mathcal A_n=\mathcal B((0,\infty)^n),\qquad
  P_{n,(a,b)}=P_{a,b}^{\otimes n},
\]
where $P_{a,b}$ has density
\begin{equation}
  p_{a,b}(x)
  =\frac{x^{a-1}e^{-x/b}}{\Gamma(a)b^a},\qquad x>0,
  \label{eq:gamma-density}
\end{equation}
with respect to Lebesgue measure, and $a$ and $b$ are the shape and scale
parameters. For observations $X_1,\ldots,X_n$, set
\[
  \ell_n(\theta)=\sum_{i=1}^n\log p_{a,b}(X_i),\qquad
  a_n=n^{-1/2},\qquad
  \Psi_n=-n^{-1}\dd\ell_n.
\]
The Fisher information per observation is
\begin{equation}
  g(a,b)=
  \begin{pmatrix}
    \psi_1(a)&b^{-1}\\
    b^{-1}&ab^{-2}
  \end{pmatrix},
  \label{eq:gamma-fisher}
\end{equation}
where $\psi_j$ is the polygamma function of order $j$. We take the
coordinate squared loss
\[
  W(\theta,z)=|\theta-z|^2,\qquad \theta,z\in\Theta,
\]
so that $m=\operatorname{id}_\Theta$.

This is the example of Sakai, Matsuda and
Kubokawa~\cite{SakaiMatsudaKubokawa2025} in which a
second-order unbiased prior does not exist.

\subsubsection{Estimators and numerical implementation}

For the gamma model, the reflected process
\eqref{eq:intro-natural-gradient-diffusion} is
\[
  \dd\theta_s
  =\frac12g(\theta_s)^{-1}\partial\ell_n(\theta_s)\,\dd s
   +\sigma(\theta_s)\,\dd W_s
   -\mathbf n(\theta_s)\,\dd K_s,
  \qquad \sigma\sigma^{\mathsf T}=g^{-1}.
\]
For the numerical implementation on the rectangular domain
$\Theta_{\mathrm{num}}$, reflection is applied facewise in the inward
$g$-normal direction.

We compare the MLE, the Jeffreys posterior mean, and $\mu_n(\theta)$, where
$\mu_n$ is an invariant probability measure of this reflected process.
The Jeffreys posterior mean is computed numerically. To approximate
$\mu_n(\theta)$, the process is initialized at the MLE and
a trajectory average is taken after a burn-in period. We use the rescaled
time $u=ns$, burn-in time $50$, averaging time $200$, and Euler--Maruyama
step size $0.005$.

We take $\theta_0=(2,1)$ and
$n\in\{50,100,200,400,800\}$. For each $n$, we first estimate the MLE
bias from $100000$ independent samples. Let $\widehat B_{\mathrm{MLE}}$
denote the resulting Monte Carlo estimate of the MLE bias.

Separately, we use $R=1000$ samples for the three estimators. Here
$\widehat\theta_{n,J}^{(r)}$, $\mu_n^{(r)}(\theta)$, and
$\widehat\theta_{n,\mathrm{MLE}}^{(r)}$ denote, respectively, the
Jeffreys posterior mean, $\mu_n(\theta)$, and the MLE computed from the
$r$th sample. We estimate their biases by
\[
  \widehat{\operatorname{Bias}}\bigl(\widehat\theta_{n,J}\bigr)
  =
  \widehat B_{\mathrm{MLE}}
  +
  \frac1R\sum_{r=1}^R
  \left(
    \widehat\theta_{n,J}^{(r)}
    -
    \widehat\theta_{n,\mathrm{MLE}}^{(r)}
  \right),
\]
and
\[
  \widehat{\operatorname{Bias}}\bigl(\mu_n(\theta)\bigr)
  =
  \widehat B_{\mathrm{MLE}}
  +
  \frac1R\sum_{r=1}^R
  \left(
    \mu_n^{(r)}(\theta)
    -
    \widehat\theta_{n,\mathrm{MLE}}^{(r)}
  \right).
\]
Because each difference is formed from estimators computed from the same
data, their common first-order fluctuation is largely cancelled, reducing
the Monte Carlo error. Parentheses in Table~\ref{tab:gamma-main} give
Monte Carlo standard errors of the displayed bias estimates.

\subsubsection{Numerical results}

Table~\ref{tab:gamma-main} gives the scaled coordinate biases and
coordinate RMSE. The scaled MLE and Jeffreys biases remain away from zero,
whereas the scaled biases of $\mu_n(\theta)$ decrease towards zero.

\begin{table}[!htbp]
\centering
\small
\setlength{\tabcolsep}{3.5pt}
\caption{Gamma model at $\theta_0=(2,1)$. Parentheses give Monte Carlo
standard errors of the displayed $n\,\mathrm{Bias}$ estimates. RMSE is
coordinatewise and is not multiplied by $n$. J denotes the Jeffreys
posterior mean.}
\label{tab:gamma-main}
\begin{tabular}{@{}cc|rrr|rrr@{}}
\toprule
&&\multicolumn{3}{c|}{$n\,\mathrm{Bias}$}
  &\multicolumn{3}{c}{$\mathrm{RMSE}$}\\
$n$&Coordinate&MLE&J&$\mu_n(\theta)$&MLE&J&$\mu_n(\theta)$\\
\midrule
50&$a$&$5.697$ (0.020)&$5.695$ (0.020)&$-0.959$ (0.057)&0.4293&0.4292&0.3831\\
&$b$&$-0.968$ (0.005)&$1.333$ (0.018)&$0.731$ (0.015)&0.2086&0.2200&0.2134\\
\addlinespace[2pt]
100&$a$&$5.563$ (0.019)&$5.563$ (0.019)&$-0.518$ (0.037)&0.2815&0.2815&0.2638\\
&$b$&$-0.971$ (0.005)&$1.269$ (0.012)&$0.367$ (0.013)&0.1448&0.1472&0.1454\\
\addlinespace[2pt]
200&$a$&$5.494$ (0.019)&$5.493$ (0.019)&$-0.240$ (0.028)&0.1924&0.1924&0.1873\\
&$b$&$-0.977$ (0.005)&$1.269$ (0.010)&$0.192$ (0.013)&0.1059&0.1074&0.1063\\
\addlinespace[2pt]
400&$a$&$5.408$ (0.019)&$5.408$ (0.019)&$-0.184$ (0.024)&0.1345&0.1345&0.1320\\
&$b$&$-0.971$ (0.005)&$1.250$ (0.007)&$0.095$ (0.012)&0.0733&0.0734&0.0732\\
\addlinespace[2pt]
800&$a$&$5.427$ (0.018)&$5.426$ (0.018)&$-0.049$ (0.022)&0.0917&0.0917&0.0911\\
&$b$&$-0.970$ (0.005)&$1.262$ (0.006)&$0.055$ (0.012)&0.0530&0.0533&0.0531\\
\bottomrule
\end{tabular}
\end{table}

At $n=800$, the scaled biases of $\mu_n(\theta)$ are close to zero
in both coordinates, while those of the MLE and Jeffreys posterior mean
remain visibly nonzero. Bias reduction does not necessarily imply a smaller
coordinate RMSE; for example, the scale RMSE of $\mu_n(\theta)$ is slightly
larger than that of the MLE at $n=50$.

\section{Proof of Theorem~\ref{thm:main-expansion}}
\label{sec:proof-main-expansion}

\subsection{Local approximation of the generator}
\label{subsec:proof-local-OU}

For a tensor field $T$, write $T_0:=T(\theta_0)$.
All norms on $T_{\theta_0}\Theta$, $T_{\theta_0}^*\Theta$, and their tensor
powers are induced by $g_0$ and $g_0^{-1}$ and are denoted by $|\cdot|$. Fr\'echet derivatives of maps defined on open subsets of $T_{\theta_0}\Theta$ are denoted by the upright symbol $\mathrm D$. For $r>0$, let
\[
  B_{T_{\theta_0}\Theta}(0,r)
  :=
  \{x\in T_{\theta_0}\Theta:|x|<r\}.
\]

Fix $r_0>0$ sufficiently small so that the exponential map at $\theta_0$ associated with the Levi--Civita
connection $\nabla$ is a diffeomorphism
\begin{equation}
  \exp_{\theta_0}:
  B_{T_{\theta_0}\Theta}(0,r_0)
  \longrightarrow
  B_\Theta(\theta_0,r_0).
  \label{eq:normal-ball-chart}
\end{equation}
Let $\log_{\theta_0}$ denote its inverse.

Throughout this subsection~\ref{subsec:proof-local-OU}, we use the normal chart
$\log_{\theta_0}$ to identify the two neighborhoods.
For an $(r,s)$-tensor field $T$, we use the same symbol for its representation in this chart and write
\[
  T(x):=(\exp_{\theta_0}^*T)_x
  \in
  (T_{\theta_0}\Theta)^{\otimes r}
  \otimes
  (T_{\theta_0}^*\Theta)^{\otimes s},
  \qquad
  x\in B_{T_{\theta_0}\Theta}(0,r_0),
\]
where $T_x(T_{\theta_0}\Theta)$ is canonically identified with
$T_{\theta_0}\Theta$. In particular, $T(0)=T_0$ and $\mathrm D^kT(x)$ denotes the $k$th
Fr\'echet derivative of this representation.

By Assumption~\textup{[A1]}, Taylor's formula, and the chain rule for the
change from the chart in \textup{[A1]} to the normal chart, for every
$p<\infty$,
\begin{equation}
\lim_{r\downarrow0}\limsup_{n\to\infty}
\left\|
  \sup_{|x|\le r}
  \sum_{j=0}^2
  \left|
    \mathrm D^j(\Psi_n-\Psi^{\theta_0})(x)
  \right|
\right\|_{L^p(P_{n,\theta_0})}
=
0.
\label{eq:A1-normal-chart}
\end{equation}

Let $\kappa\in(0,1)$ be as in Assumption~\textup{[A2]}$(\mu_n)$.
Decreasing $r_0$ if necessary and then choosing $n_0\in\mathbb N$
sufficiently large, we may assume that
\begin{equation}
  \sup_{n\ge n_0}
  \left\|
    \sup_{|x|\le r_0}
    \sum_{j=0}^2
    \left|
      \mathrm D^j(\Psi_n-\Psi^{\theta_0})(x)
    \right|
  \right\|_{L^2(P_{n,\theta_0})}
  <\infty,
  \qquad
  a_{n_0}^\kappa<r_0.
  \label{eq:A1-fixed-normal-bound}
\end{equation}
For $n\ge n_0$, set
\[
  O_n:=B_\Theta(\theta_0,a_n^\kappa),
\]
and define
\begin{equation}
  \theta_n(u):=\exp_{\theta_0}(a_nu),
  \qquad
  |u|\le a_n^{\kappa-1}.
  \label{eq:theta-n-definition}
\end{equation}

Fix $\chi\in C_c^\infty(B_{T_{\theta_0}\Theta}(0,r_0))$ such that $\chi=1$ on a neighborhood of $\overline{B_{T_{\theta_0}\Theta}(0,a_{n_0}^\kappa)}$, and define
$\bar x:\Theta\to T_{\theta_0}\Theta$ by
\begin{equation}
  \bar x(\theta)
  :=
  \begin{cases}
    \chi(\log_{\theta_0}\theta)\log_{\theta_0}\theta,
    &\theta\in B_\Theta(\theta_0,r_0),\\
    0,
    &\theta\notin B_\Theta(\theta_0,r_0),
  \end{cases}
  \qquad
  \bar u_n:=a_n^{-1}\bar x.
  \label{eq:barx-definition}
\end{equation}
Then $  \supp\bar x
  \Subset
  B_\Theta(\theta_0,r_0)
  \Subset
  \Theta^\circ$ and
\begin{equation}
  \bar u_n(\theta_n(u))=u,
  \qquad
  n\ge n_0,\quad |u|\le a_n^{\kappa-1}.
  \label{baru=u}
\end{equation}
\begin{proposition}
\label{prop:local-OU-reduction}
Under Assumption~\textup{[A1]} and Assumption~\textup{[A2]}$(\mu_n)$, for every $n\ge n_0$,
every smooth $\varphi:T_{\theta_0}\Theta\to\mathbb R$, and every
$|u|\le a_n^{\kappa-1}$,
\begin{equation}
  a_n^2L_n(\varphi\circ\bar u_n)(\theta_n(u))
  =
  \frakL_n^{[0]}\varphi(u)
  +\frakL_n^{[1]}\varphi(u)
  +\mathcal R_n\varphi(u),
  \label{eq:local-OU-generator-decomposition}
\end{equation}
where
\begin{align}
  \frakL_n^{[0]}\varphi(u)
  &:=
  \frac{\tau}{2}\mathrm D^2\varphi(u)[g_0^{-1}]
  +\frac12\mathrm D\varphi(u)[\zeta_{n,0}-u],
  \label{eq:leading-OU-operator}
  \\
  \frakL_n^{[1]}\varphi(u)
  &:=
  \frac12\mathrm D\varphi(u)\!\left[
    \Upsilon_{n,0}[u]
    +a_nV_0
    +a_n(\nabla-D^{(-1)})_0
      \left[u,\zeta_{n,0}-\frac12u\right]
  \right].
  \label{eq:first-OU-correction}
\end{align}
There exists $\epsilon_n=o_{\Linftyminus(\theta_0)}(1)$ such that,
uniformly on $|u|\le a_n^{\kappa-1}$,
\begin{equation}
  |\mathcal R_n\varphi(u)|
  \le
  a_n\epsilon_n(1+|u|^2)
  \left\{
    |\mathrm D\varphi(u)|+|\mathrm D^2\varphi(u)|
  \right\}.
  \label{eq:local-OU-remainder}
\end{equation}
\end{proposition}

\begin{proof}
For $|u|\le a_n^{\kappa-1}$, by \eqref{baru=u},
\begin{equation}
\begin{aligned}
&a_n^2L_n(\varphi\circ\bar u_n)(\theta_n(u))
\\
&\quad=
\frac{\tau}{2}\mathrm D^2\varphi(u)[g^{-1}(a_nu)]
\\ & \quad \quad +\frac12\mathrm D\varphi(u)\Big[
  -a_n^{-1}(g^{-1}\Psi_n)(a_nu)
  +a_nV(a_nu)
  -\tau a_n\Gamma^\nabla(a_nu)[g^{-1}(a_nu)]
\Big],
\end{aligned}
\label{eq:normal-coordinate-generator}
\end{equation}
where $\Gamma^\nabla$ denotes the Christoffel symbol of $\nabla$ in this chart.

For $|u|\le a_n^{\kappa-1}$, Taylor's formula gives
\begin{equation}
\begin{aligned}
-a_n^{-1}(g^{-1}\Psi_n)(a_nu)
={}&
-a_n^{-1}(g^{-1}\Psi_n)_0
-\mathrm D(g^{-1}\Psi_n)(0)[u]
\\
&-\frac{a_n}{2}
  \mathrm D^2(g^{-1}\Psi^{\theta_0})(0)[u,u]
+r_n(u),
\end{aligned}
\label{eq:local-drift-Taylor}
\end{equation}
where
\[
\begin{aligned}
r_n(u)
:={}&
-a_n\int_0^1(1-t)
\Big\{
  \mathrm D^2(g^{-1}\Psi_n)(ta_nu)
  -\mathrm D^2(g^{-1}\Psi^{\theta_0})(0)
\Big\}[u,u] \,\dd t .
\end{aligned}
\]
We first establish the following relations:
\begin{align}
  \mathrm D(g^{-1}\Psi_n)(0)[u]
  &=
  u-\Upsilon_{n,0}[u]
  -a_n(\nabla-D^{(-1)})_0[u,\zeta_{n,0}],
  \label{eq:local-drift-first-jet}
  \\
  \mathrm D^2(g^{-1}\Psi^{\theta_0})(0)[u,u]
  &=
  (\nabla-D^{(-1)})_0[u,u],
  \label{eq:local-drift-second-jet}
  \\
  \sup_{|x|\le a_n^\kappa}
  \left|
    \mathrm D^2(g^{-1}\Psi_n)(x)
    -
    \mathrm D^2(g^{-1}\Psi^{\theta_0})(0)
  \right|
  &=
  o_{\Linftyminus(\theta_0)}(1).
  \label{eq:local-drift-second-jet-control}
\end{align}

We first prove \eqref{eq:local-drift-first-jet}. Extend
$u\in T_{\theta_0}\Theta$ to a vector field $U$. Since $\log_{\theta_0}$ is $\nabla$-normal at $\theta_0$,
\[
  \mathrm D(g^{-1}\Psi_n)(0)[u]
  =
  \bigl(\nabla_{U}(g^{-1}\Psi_n)\bigr)_0.
\]
By the duality of $D^{(1)}$ and $D^{(-1)}$, for local vector fields
$X,Y$,
\begin{align*}
  \left\langle
    D^{(-1)}_X(g^{-1}\Psi_n),Y
  \right\rangle
  &=
  X\left\langle g^{-1}\Psi_n,Y\right\rangle
  -
  \left\langle g^{-1}\Psi_n,D^{(1)}_X Y\right\rangle
\\
  &=
  X\{\Psi_n[Y]\}
  -
  \Psi_n[D^{(1)}_X Y]
\\
  &=
  (D^{(1)}_X\Psi_n)[Y].
\end{align*}
By the definition of $\Upsilon_n$ in \eqref{eq:Z-definition},
\[
  (D^{(1)}_X\Psi_n)[Y]
  =
  \langle X-\Upsilon_n[X],Y\rangle.
\]
Since $Y$ is arbitrary,
\begin{equation}
  D^{(-1)}_X(g^{-1}\Psi_n)
  =
  X-\Upsilon_n[X].
  \label{eq:Dminus1-Psi-sharp}
\end{equation}
Using \eqref{eq:Dminus1-Psi-sharp} and
$(g^{-1}\Psi_n)_0=-a_n\zeta_{n,0}$, we obtain
\begin{align*}
  \mathrm D(g^{-1}\Psi_n)(0)[u]
  &=
  \bigl(D^{(-1)}_{U}(g^{-1}\Psi_n)\bigr)_0
  +
  (\nabla-D^{(-1)})_0
  [u,(g^{-1}\Psi_n)_0]
\\
  &=
  u-\Upsilon_{n,0}[u]
  -a_n(\nabla-D^{(-1)})_0[u,\zeta_{n,0}],
\end{align*}
which proves \eqref{eq:local-drift-first-jet}.

We next prove \eqref{eq:local-drift-second-jet}. Let
$\widetilde u\in T_{\theta_0}\Theta$ and extend $u,\widetilde u$ to
the constant vector fields $U,\widetilde U$ in the normal chart, respectively. Since
$\Psi^{\theta_0}(\theta_0)=0$, we have
$(g^{-1}\Psi^{\theta_0})_0=0$, while normality of the chart gives
$\mathrm Dg(0)=0$. Hence
\begin{align}
  \left\langle
    \mathrm D^2(g^{-1}\Psi^{\theta_0})(0)[u,u],\widetilde u
  \right\rangle
  &=
  UU\{\Psi^{\theta_0}[\widetilde U]\}(\theta_0)
  \notag\\
  &=
  \Psi(UU\widetilde U\mid)(\theta_0).
  \label{eq:second-jet-precontrast}
\end{align}
By \eqref{eq:precontrast-metric},
\[
  \Psi(U\widetilde U\mid)=\langle \widetilde U,U\rangle.
\]
Differentiating this identity along the diagonal in the direction $U$
and evaluating at $\theta_0$ gives
\[
  \Psi(UU\widetilde U\mid)(\theta_0)
  +
  \Psi(U\widetilde U\mid U)(\theta_0)
  =
  U\langle \widetilde U,U\rangle(\theta_0).
\]
Since $U,\widetilde U$ are constant in the normal chart, $U\langle \widetilde U,U\rangle(\theta_0)=0$. Therefore,
by \eqref{eq:D1-definition},
\[
  \Psi(UU\widetilde U\mid)(\theta_0)
  =
  -\Psi(U\widetilde U\mid U)(\theta_0)
  =
  \left\langle (D^{(1)}_{U}\widetilde U)_0,u\right\rangle.
\]
The duality relation \eqref{eq:precontrast-duality} and $\nabla U =0$ give
\[
\begin{aligned}
  \left\langle (D^{(1)}_{U}\widetilde U)_0,u\right\rangle
  &=
  -\left\langle \widetilde u,(D^{(-1)}_{U}U)_0\right\rangle
\\
  &=
  \left\langle
    (\nabla-D^{(-1)})_0[u,u],\widetilde u
  \right\rangle.
\end{aligned}
\]
Combining this with \eqref{eq:second-jet-precontrast} and using the
arbitrariness of $\widetilde u$ proves \eqref{eq:local-drift-second-jet}.

We finally prove \eqref{eq:local-drift-second-jet-control}.
Since $a_n^\kappa\to0$, \eqref{eq:A1-normal-chart} and the product rule give
\[
  \sup_{|x|\le a_n^\kappa}
  \left|
    \mathrm D^2(g^{-1}\Psi_n)(x)
    -
    \mathrm D^2(g^{-1}\Psi^{\theta_0})(x)
  \right|
  =
  o_{\Linftyminus(\theta_0)}(1).
\]
Since $g^{-1}\Psi^{\theta_0}$ is smooth,
\[
  \sup_{|x|\le a_n^\kappa}
  \left|
    \mathrm D^2(g^{-1}\Psi^{\theta_0})(x)
    -
    \mathrm D^2(g^{-1}\Psi^{\theta_0})(0)
  \right|
  =
  o(1).
\]
This proves \eqref{eq:local-drift-second-jet-control}.

Combining \eqref{eq:local-drift-Taylor},
\eqref{eq:local-drift-first-jet},
\eqref{eq:local-drift-second-jet}, and
\eqref{eq:local-drift-second-jet-control}, together with the definition of
$\zeta_n$, yields
\begin{equation}
\begin{aligned}
  -a_n^{-1}(g^{-1}\Psi_n)(a_nu)
  ={}&
  \zeta_{n,0}-u+\Upsilon_{n,0}[u]
\\
  &+
  a_n(\nabla-D^{(-1)})_0
  \left[u,\zeta_{n,0}-\frac12u\right]
  +r_n(u),
\end{aligned}
\label{eq:local-drift-expansion}
\end{equation}
where
\begin{equation}
  \sup_{|u|\le a_n^{\kappa-1}}
  \frac{|r_n(u)|}{1+|u|^2}
  =
  o_{\Linftyminus(\theta_0)}(a_n).
  \label{eq:local-drift-remainder}
\end{equation}

Finally, normal coordinates give
\[
  g^{-1}(x)=g_0^{-1}+O(|x|^2),
  \qquad
  \Gamma^\nabla(x)=O(|x|),
  \qquad
  V(x)=V_0+O(|x|).
\]
Substituting these estimates and \eqref{eq:local-drift-expansion} into
\eqref{eq:normal-coordinate-generator} gives
\[
\begin{aligned}
  |\mathcal R_n\varphi(u)|
  \le{}&
  \frac12|r_n(u)|\,|\mathrm D\varphi(u)|
\\
  &+
  Ca_n^2(1+|u|^2)
  \left\{
    |\mathrm D\varphi(u)|+|\mathrm D^2\varphi(u)|
  \right\},
\end{aligned}
\]
uniformly on $|u|\le a_n^{\kappa-1}$. Together with
\eqref{eq:local-drift-remainder}, this proves
\eqref{eq:local-OU-remainder}.
\end{proof}
\begin{proposition}
\label{prop:OU-scale-moment-bound}
Under Assumption~\textup{[A1]} and Assumption~\textup{[A2]}$(\mu_n)$, let
$q_n:T_{\theta_0}\Theta\to\R$ be random polynomials of degree at most
$M$ whose coefficients are $O_{\Linftyminus(\theta_0)}(1)$, in the sense that
\[
  \mathrm D^j q_n(0)
  =
  O_{\Linftyminus(\theta_0)}(1),
  \qquad
  j=0,\ldots,M.
\]
Then
\begin{align}
  \int_\Theta |q_n(\bar u_n)|\,\dd\mu_n
  &=
  O_{\Linftyminus(\theta_0)}(1),
  \label{eq:globalized-OU-moment-bound}
  \\
  \int_{O_n^c}|q_n(\bar u_n)|\,\dd\mu_n
  &=
  O_{\Linftyminus(\theta_0)}(a_n^L),
  \label{eq:polynomial-tail-bound}
  \\
  \int_{O_n}
  a_n^2L_n(q_n\circ\bar u_n)(\theta)\,\mu_n(\dd\theta)
  &=
  O_{\Linftyminus(\theta_0)}(a_n^L)
  \label{eq:polynomial-generator-tail}
\end{align}
for every $L>0$. If the coefficients of $q_n$ are
$o_{\Linftyminus(\theta_0)}(1)$, then
\eqref{eq:globalized-OU-moment-bound} holds with
$o_{\Linftyminus(\theta_0)}(1)$ in place of
$O_{\Linftyminus(\theta_0)}(1)$.
\end{proposition}

\begin{proof}
Throughout the proof, $c,C>0$ denote deterministic constants independent
of $n$, whose values may change from line to line.

We first prove \eqref{eq:polynomial-generator-tail}. Set
\[
  C_n
  :=
  \sum_{j=0}^M|\mathrm D^jq_n(0)|
  =
  O_{\Linftyminus(\theta_0)}(1).
\]
By polynomiality, for $j=1,2$,
\begin{equation}
  |\mathrm D^jq_n(u)|
  \le
  CC_n(1+|u|)^{M-j}.
  \label{eq:polynomial-derivative-bound}
\end{equation}
Fix $p\ge1$. By \eqref{eq:A1-normal-chart}, there exists
$r_p\in(0,r_0)$ such that, for all sufficiently large $n$,
\[
  \left\|
    \sup_{|x|\le r_p}
    \sum_{j=0}^2
    \left|
      \mathrm D^j(\Psi_n-\Psi^{\theta_0})(x)
    \right|
  \right\|_{L^{4p}(P_{n,\theta_0})}
  =
  O(1).
\]
Choose $\chi_p\in C_c^\infty(B_{T_{\theta_0}\Theta}(0,r_p))$
such that $\chi_p=1$ on
$B_{T_{\theta_0}\Theta}(0,r_p/2)$, and define
\[
  \bar x_p(\theta)
  :=
  \begin{cases}
    \chi_p(\log_{\theta_0}\theta)\log_{\theta_0}\theta,
    &\theta\in B_\Theta(\theta_0,r_p),\\
    0,
    &\theta\notin B_\Theta(\theta_0,r_p),
  \end{cases}
  \qquad
  \bar u_{n,p}:=a_n^{-1}\bar x_p.
\]
Thus $\bar x_p$ is obtained from the right-hand side of
\eqref{eq:barx-definition} by replacing $\chi$ with $\chi_p$.
Since $a_n^\kappa\to0$, for all sufficiently large $n$,
\[
  \bar u_{n,p}=\bar u_n
  \qquad\text{on }O_n.
\]
Moreover, by the preceding $L^{4p}$ bound and the smoothness of
$\Psi^{\theta_0}$,
\begin{equation}
  \sup_{\theta\in\supp\bar x_p}|\Psi_n(\theta)|
  =
  O_{L^{4p}(P_{n,\theta_0})}(1).
  \label{eq:Psi-support-bound}
\end{equation}
Since $\bar x_p$ is a fixed smooth map with compact support in
$\Theta^\circ$, the chain rule and
\eqref{eq:polynomial-derivative-bound} give
\[
  \sup_{\theta\in\Theta}
  \left\{
    |\dd(q_n\circ\bar u_{n,p})(\theta)|
    +
    |\Delta(q_n\circ\bar u_{n,p})(\theta)|
  \right\}
  \le
  CC_na_n^{-M}.
\]
Together with \eqref{eq:Psi-support-bound} and
\eqref{eq:full-generator}, H\"older's inequality yields
\begin{equation}
  \left\|
    \sup_{\theta\in\Theta}
    \left|
      a_n^2L_n(q_n\circ\bar u_{n,p})(\theta)
    \right|
  \right\|_{L^{2p}(P_{n,\theta_0})}
  =
  O(a_n^{-M}).
  \label{eq:polynomial-generator-rough-bound}
\end{equation}

By
$q_n(\bar u_{n,p})-q_n(0)\in C_c^\infty(\Theta^\circ)$,
\[
  \int_\Theta
  a_n^2L_n(q_n\circ\bar u_{n,p})(\theta)\,\mu_n(\dd\theta)
  =
  0.
\]
Hence, for all sufficiently large $n$,
\[
\begin{aligned}
  \left|
    \int_{O_n}
    a_n^2L_n(q_n\circ\bar u_n)(\theta)\,\mu_n(\dd\theta)
  \right|
  &=
  \left|
    \int_{O_n^c}
    a_n^2L_n(q_n\circ\bar u_{n,p})(\theta)\,\mu_n(\dd\theta)
  \right|
  \\
  &\le
  \mu_n(O_n^c)
  \sup_{\theta\in\Theta}
  \left|
    a_n^2L_n(q_n\circ\bar u_{n,p})(\theta)
  \right|.
\end{aligned}
\]
Since $0\le\mu_n(O_n^c)\le1$, Assumption~\textup{[A2]}$(\mu_n)$
implies, for every $K>0$,
\[
  \|\mu_n(O_n^c)\|_{L^{2p}(P_{n,\theta_0})}
  \le
  \{\E_{n,\theta_0}[\mu_n(O_n^c)]\}^{1/(2p)}
  =
  O(a_n^K).
\]
Taking $K=L+M$ and applying H\"older's inequality together with
\eqref{eq:polynomial-generator-rough-bound} gives
\[
  \left\|
    \int_{O_n}
    a_n^2L_n(q_n\circ\bar u_n)(\theta)\,\mu_n(\dd\theta)
  \right\|_{L^p(P_{n,\theta_0})}
  =
  O(a_n^L).
\]
Since $p<\infty$ is arbitrary, this proves
\eqref{eq:polynomial-generator-tail}.

We next prove the moment bound. Set
\[
  w(u):=(1+|u|^2)^M.
\]
We first show that
\begin{equation}
  \int_\Theta w(\bar u_n)\,\dd\mu_n
  =
  O_{\Linftyminus(\theta_0)}(1).
  \label{eq:OU-moment-bound}
\end{equation}
For $j=1,2$,
\begin{equation}
  |\mathrm D^jw(u)|
  \le
  C(1+|u|^2)^{M-j/2}.
  \label{eq:w-derivative-bound}
\end{equation}
Moreover,
\begin{equation}
  \sup_{\theta\in\Theta}
  w(\bar u_n(\theta))
  \le
  Ca_n^{-2M}.
  \label{eq:w-global-rough-bound}
\end{equation}
Since $w$ is a deterministic polynomial of degree $2M$, the estimate already
proved in \eqref{eq:polynomial-generator-tail}, with $2M$ in place of $M$,
gives
\begin{equation}
  \int_{O_n}
  a_n^2L_n(w\circ\bar u_n)(\theta)\,\mu_n(\dd\theta)
  =
  O_{\Linftyminus(\theta_0)}(a_n^L)
  \qquad
  \text{for every }L>0.
  \label{eq:OU-local-stationarity-error}
\end{equation}

Direct differentiation and Young's inequality give
\[
  \frakL_n^{[0]}w(u)
  \le
  -cw(u)
  +C(1+|\zeta_{n,0}|^{2M}).
\]
By Proposition~\ref{prop:local-OU-reduction} and
\eqref{eq:w-derivative-bound}, there exists a nonnegative random variable
$\eta_n$ such that
\[
\begin{aligned}
  a_n^2L_n(w\circ\bar u_n)(\theta_n(u))
  \le{}&
  -(c-\eta_n)w(u)
  +C(1+|\zeta_{n,0}|^{2M}),
\\[-1mm]
&\hspace{38mm}|u|\le a_n^{\kappa-1},
\end{aligned}
\]
where
\[
  \eta_n
  \le
  C\left\{
    |\Upsilon_{n,0}|
    +a_n(1+|\zeta_{n,0}|)
    +a_n^\kappa
    +(a_n+a_n^\kappa)|\epsilon_n|
  \right\}.
\]
Here $\epsilon_n=o_{\Linftyminus(\theta_0)}(1)$ is as in
\eqref{eq:local-OU-remainder}. By \eqref{eq:zeta,Z-order},
\[
  \zeta_{n,0}=O_{\Linftyminus(\theta_0)}(1),
  \qquad
  \Upsilon_{n,0}=o_{\Linftyminus(\theta_0)}(a_n^{1/2}),
\]
and therefore
\[
  \eta_n
  =
  O_{\Linftyminus(\theta_0)}
  \bigl(a_n^{\min\{\kappa,1/2\}}\bigr).
\]
Set $G_n:=\{\eta_n\le c/2\}$. Markov's inequality gives
\begin{equation}
  P_{n,\theta_0}(G_n^c)
  =
  O(a_n^L)
  \qquad
  \text{for every }L>0.
  \label{eq:OU-good-event}
\end{equation}

On $G_n$,
\[
  a_n^2L_n(w\circ\bar u_n)(\theta)
  \le
  -\frac c2w(\bar u_n(\theta))
  +C(1+|\zeta_{n,0}|^{2M}),
  \qquad
  \theta\in O_n.
\]
Hence
\[
\begin{aligned}
  \frac c2\mathbf 1_{G_n}
  \int_{O_n}w(\bar u_n)\,\dd\mu_n
  \le{}&
  C\mathbf 1_{G_n}(1+|\zeta_{n,0}|^{2M})
  \\
  &-
  \mathbf 1_{G_n}
  \int_{O_n}
  a_n^2L_n(w\circ\bar u_n)\,\dd\mu_n .
\end{aligned}
\]
Taking expectations, using \eqref{eq:OU-local-stationarity-error}, and
using $\zeta_{n,0}=O_{\Linftyminus(\theta_0)}(1)$ gives
\begin{equation}
  \E_{n,\theta_0}\left[
    \mathbf 1_{G_n}
    \int_{O_n}w(\bar u_n)\,\dd\mu_n
  \right]
  =
  O(1).
  \label{eq:OU-moment-good}
\end{equation}
By \eqref{eq:w-global-rough-bound}, Assumption~\textup{[A2]}$(\mu_n)$, and
\eqref{eq:OU-good-event},
\[
  \E_{n,\theta_0}\left[
    \int_{O_n^c}w(\bar u_n)\,\dd\mu_n
  \right]
  =
  O(1),
\]
and
\[
  \E_{n,\theta_0}\left[
    \mathbf 1_{G_n^c}
    \int_\Theta w(\bar u_n)\,\dd\mu_n
  \right]
  =
  O(1).
\]
Together with \eqref{eq:OU-moment-good}, this yields
\[
  \E_{n,\theta_0}\left[
    \int_\Theta w(\bar u_n)\,\dd\mu_n
  \right]
  =
  O(1).
\]
For every $p\in\mathbb N$, repeating the preceding argument with $M$
replaced by $Mp$ and using Jensen's inequality give
\[
\begin{aligned}
  \E_{n,\theta_0}\left[
    \left\{
      \int_\Theta w(\bar u_n)\,\dd\mu_n
    \right\}^p
  \right]
  &\le
  \E_{n,\theta_0}\left[
    \int_\Theta
    (1+|\bar u_n|^2)^{Mp}\,\dd\mu_n
  \right]
  \\
  &=
  O(1).
\end{aligned}
\]
This proves \eqref{eq:OU-moment-bound}.

Taylor's formula gives
\[
  |q_n(u)|
  \le
C  C_nw(u).
\]
Hence
\[
  \int_\Theta |q_n(\bar u_n)|\,\dd\mu_n
  \le
  CC_n\int_\Theta w(\bar u_n)\,\dd\mu_n
  =
  O_{\Linftyminus(\theta_0)}(1),
\]
which proves \eqref{eq:globalized-OU-moment-bound}. The same estimate gives
$o_{\Linftyminus(\theta_0)}(1)$ when the coefficients of $q_n$ are
$o_{\Linftyminus(\theta_0)}(1)$. Moreover,
$|\bar u_n|\le Ca_n^{-1}$, so
\[
\begin{aligned}
  \int_{O_n^c}|q_n(\bar u_n)|\,\dd\mu_n
  &\le
C  C_n\int_{O_n^c}w(\bar u_n)\,\dd\mu_n
  \\
  &\le
  CC_na_n^{-2M}\mu_n(O_n^c)
  =
  O_{\Linftyminus(\theta_0)}(a_n^L)
\end{aligned}
\]
for every $L>0$, by Assumption~\textup{[A2]}$(\mu_n)$ and H\"older's inequality. This proves \eqref{eq:polynomial-tail-bound}.
\end{proof}

\subsection{Global expansion of invariant averages}
\label{subsec:proof-OU-Poisson}

Let $N$ be a $T_{\theta_0}\Theta$-valued random variable independent of
$\Psi_n$ with law
\[
  N(0,\tau g_0^{-1});
\]
equivalently, for every $\alpha\in T_{\theta_0}^*\Theta$,
\[
  \alpha[N]\sim
  N\!\left(0,\tau\braket{\alpha,\alpha}\right),
\]
where, when $\tau=0$, the law is understood as $\delta_0$.
We write $\E_N$ for expectation with respect to $N$ only.

Consider the centered Ornstein--Uhlenbeck operator
\begin{equation}
  \frakL^{[0]}\varphi(u)
  :=
  \frac{\tau}{2}\mathrm D^2\varphi(u)[g_0^{-1}]
  -\frac12\mathrm D\varphi(u)[u].
  \label{eq:centered-OU-operator}
\end{equation}
By standard Ornstein--Uhlenbeck theory (with the case $\tau=0$ being
immediate), there exists a unique linear operator $\calP$ on polynomials
on $T_{\theta_0}\Theta$ such that
\begin{equation}
  -\frakL^{[0]}\calP q
  =q-\E_N[q(N)],
  \qquad
  \E_N[(\calP q)(N)]=0
  \label{eq:centered-OU-Poisson-equation}
\end{equation}
for every polynomial $q$. By the definition, $\calP$ preserves polynomial degree.

For a polynomial $q$, define
\[
  (\calP_nq)(u)
  :=
  \left[
    \calP\{q(\,\cdot+\zeta_{n,0})\}
  \right](u-\zeta_{n,0}).
\]
By \eqref{eq:leading-OU-operator} and \eqref{eq:centered-OU-operator},
\begin{equation}
  \frakL_n^{[0]}q(u)
  =
  \left[
    \frakL^{[0]}\{q(\,\cdot+\zeta_{n,0})\}
  \right](u-\zeta_{n,0}).
  \label{eq:OU-shift-identity}
\end{equation}
Set
\begin{equation}
  N_n:=\zeta_{n,0}+N.
  \label{eq:leading-OU-law}
\end{equation}
It follows from \eqref{eq:centered-OU-Poisson-equation} and
\eqref{eq:OU-shift-identity} that
\begin{equation}
  -\frakL_n^{[0]}\calP_nq
  =q-\E_N[q(N_n)],
  \qquad
  \E_N[(\calP_nq)(N_n)]=0.
  \label{eq:OU-Poisson-equation}
\end{equation}

Recall from \eqref{eq:zeta,Z-order} that \(
  \zeta_{n,0}=O_{\Linftyminus(\theta_0)}(1)\) and \( \Upsilon_{n,0}=o_{\Linftyminus(\theta_0)}(a_n^{1/2})
\). Since $\calP$ preserves polynomial degree, $\calP_n$ preserves polynomial
degree and coefficient order: if $q_n$ has degree at most $M$ and all its
coefficients are $O_{\Linftyminus(\theta_0)}(1)$
(resp.\ $o_{\Linftyminus(\theta_0)}(1)$), then $\calP_nq_n$ has degree at
most $M$ and all its coefficients are
$O_{\Linftyminus(\theta_0)}(1)$
(resp.\ $o_{\Linftyminus(\theta_0)}(1)$).

For polynomials with values in a finite-dimensional vector space, we let
$\calP$, $\calP_n$, $\frakL^{[0]}$, $\frakL_n^{[0]}$, and
$\frakL_n^{[1]}$ act componentwise with respect to any basis.

\begin{proposition}
\label{prop:polynomial-Poisson-expansion}
Under Assumption~\textup{[A1]} and Assumption~\textup{[A2]}$(\mu_n)$, let
$q_n:T_{\theta_0}\Theta\to\R$ be random polynomials of degree at most
$M$ whose coefficients are $O_{\Linftyminus(\theta_0)}(1)$. Then
\begin{equation}
\begin{aligned}
  \int_\Theta q_n(\bar u_n)\,\dd\mu_n
  ={}&
  \E_N[q_n(N_n)]
  +\E_N\!\left[
    (\frakL_n^{[1]}\calP_nq_n)(N_n)
  \right]
  +o_{\Linftyminus(\theta_0)}(a_n).
  \label{eq:polynomial-Poisson-expansion}
\end{aligned}
\end{equation}
The same conclusion holds, under the same degree and coefficient-order
assumptions, for polynomials with values in any finite-dimensional vector space.
\end{proposition}

\begin{proof}
The coefficients of $\calP_nq_n$ are
$O_{\Linftyminus(\theta_0)}(1)$ by the observation above. First,
Proposition~\ref{prop:OU-scale-moment-bound} gives
\begin{equation}
  \int_\Theta q_n(\bar u_n)\,\dd\mu_n
  =
  \int_{O_n}q_n(\bar u_n)\,\dd\mu_n
  +o_{\Linftyminus(\theta_0)}(a_n).
  \label{eq:Poisson-localization}
\end{equation}
For $\theta\in O_n$, Proposition~\ref{prop:local-OU-reduction} and
\eqref{eq:OU-Poisson-equation} give
\[
\begin{aligned}
  q_n(\bar u_n(\theta))
  ={}&
  \E_N[q_n(N_n)]
  -
  a_n^2L_n\{(\calP_nq_n)\circ\bar u_n\}(\theta)
  \\
  &+
  (\frakL_n^{[1]}\calP_nq_n)(\bar u_n(\theta))
  +
  (\mathcal R_n\calP_nq_n)(\bar u_n(\theta)).
\end{aligned}
\]
Integrating this identity over $O_n$ and using
\eqref{eq:Poisson-localization} yields
\begin{align}
  \int_\Theta q_n(\bar u_n)\,\dd\mu_n
  &=
  \E_N[q_n(N_n)]\mu_n(O_n)
  -
  \int_{O_n}
    a_n^2L_n\{(\calP_nq_n)\circ\bar u_n\}(\theta)\,
    \mu_n(\dd\theta)
  \notag\\
  &\quad+
  \int_{O_n}
    (\frakL_n^{[1]}\calP_nq_n)(\bar u_n(\theta))\,
    \mu_n(\dd\theta)
  \notag\\
  &\quad+
  \int_{O_n}
    (\mathcal R_n\calP_nq_n)(\bar u_n(\theta))\,
    \mu_n(\dd\theta)
  +o_{\Linftyminus(\theta_0)}(a_n).
  \label{eq:Poisson-first-decomposition}
\end{align}

Proposition~\ref{prop:OU-scale-moment-bound} gives, for every $L>0$,
\begin{equation}
  \int_{O_n}
  a_n^2L_n\{(\calP_nq_n)\circ\bar u_n\}(\theta)\,
  \mu_n(\dd\theta)
  =
  O_{\Linftyminus(\theta_0)}(a_n^L).
  \label{eq:Poisson-generator-tail}
\end{equation}
By \eqref{eq:local-OU-remainder} and
Proposition~\ref{prop:OU-scale-moment-bound},
\begin{equation}
  \int_{O_n}
  \left|
    (\mathcal R_n\calP_nq_n)(\bar u_n(\theta))
  \right|
  \mu_n(\dd\theta)
  =
  o_{\Linftyminus(\theta_0)}(a_n).
  \label{eq:Poisson-remainder-integral}
\end{equation}
Since $\E_N[q_n(N_n)]=O_{\Linftyminus(\theta_0)}(1)$, Assumption~\textup{[A2]}$(\mu_n)$ gives
\begin{equation}
  \E_N[q_n(N_n)]\mu_n(O_n)
  =
  \E_N[q_n(N_n)]
  +o_{\Linftyminus(\theta_0)}(a_n).
  \label{eq:Poisson-mass-replacement}
\end{equation}
By \eqref{eq:first-OU-correction}, the bounds
\[
  \zeta_{n,0}=O_{\Linftyminus(\theta_0)}(1),
  \qquad
  \Upsilon_{n,0}=o_{\Linftyminus(\theta_0)}(a_n^{1/2})
\]
from \eqref{eq:zeta,Z-order}, and the coefficient bound for
$\calP_nq_n$, the coefficients of
$\frakL_n^{[1]}\calP_nq_n$ are
$o_{\Linftyminus(\theta_0)}(a_n^{1/2})$. Hence
Proposition~\ref{prop:OU-scale-moment-bound} gives
\begin{equation}
  \int_{O_n}
    (\frakL_n^{[1]}\calP_nq_n)(\bar u_n)\,\dd\mu_n
  =
  \int_\Theta
    (\frakL_n^{[1]}\calP_nq_n)(\bar u_n)\,\dd\mu_n
  +o_{\Linftyminus(\theta_0)}(a_n).
  \label{eq:Poisson-first-correction-localization}
\end{equation}
Combining \eqref{eq:Poisson-generator-tail},
\eqref{eq:Poisson-remainder-integral},
\eqref{eq:Poisson-mass-replacement}, and
\eqref{eq:Poisson-first-correction-localization} with
\eqref{eq:Poisson-first-decomposition} yields
\begin{equation}
  \int_\Theta q_n(\bar u_n)\,\dd\mu_n
  =
  \E_N[q_n(N_n)]
  +
  \int_\Theta
    (\frakL_n^{[1]}\calP_nq_n)(\bar u_n)\,\dd\mu_n
  +o_{\Linftyminus(\theta_0)}(a_n).
  \label{eq:Poisson-first-iteration}
\end{equation}

The preceding one-step identity applies again with $q_n$ replaced by
$\frakL_n^{[1]}\calP_nq_n$. Hence
\[
\begin{aligned}
  &\int_\Theta
    (\frakL_n^{[1]}\calP_nq_n)(\bar u_n)\,\dd\mu_n
  \\
  &\quad=
  \E_N\!\left[
    (\frakL_n^{[1]}\calP_nq_n)(N_n)
  \right]
  +
  \int_\Theta
    \left[
      \frakL_n^{[1]}
      \calP_n\{\frakL_n^{[1]}\calP_nq_n\}
    \right](\bar u_n)\,\dd\mu_n
  +o_{\Linftyminus(\theta_0)}(a_n).
\end{aligned}
\]
By \eqref{eq:first-OU-correction}, the bounds
\[
  \zeta_{n,0}=O_{\Linftyminus(\theta_0)}(1),
  \qquad
  \Upsilon_{n,0}=o_{\Linftyminus(\theta_0)}(a_n^{1/2})
\]
from \eqref{eq:zeta,Z-order}, and the coefficient-order preservation of
$\calP_n$, the polynomial
\[
  \frakL_n^{[1]}
  \calP_n\{\frakL_n^{[1]}\calP_nq_n\}
\]
has uniformly bounded degree and
$o_{\Linftyminus(\theta_0)}(a_n)$ coefficients. Applying
Proposition~\ref{prop:OU-scale-moment-bound} to this polynomial divided by
$a_n$ gives
\[
  \int_\Theta
  \left|
    \left[
      \frakL_n^{[1]}
      \calP_n\{\frakL_n^{[1]}\calP_nq_n\}
    \right](\bar u_n)
  \right|\,\dd\mu_n
  =
  o_{\Linftyminus(\theta_0)}(a_n).
\]
Substituting this into \eqref{eq:Poisson-first-iteration} proves
\eqref{eq:polynomial-Poisson-expansion}. For a polynomial with values in
a finite-dimensional vector space, choose a basis of the target
space and apply the scalar-valued result to each component.
\end{proof}

We next apply Proposition~\ref{prop:polynomial-Poisson-expansion} to the identity and quadratic polynomials
\[
T_{\theta_0} \Theta  \ni  u\longmapsto u \in T_{\theta_0} \Theta  ,
  \qquad
T_{\theta_0} \Theta  \ni    u\longmapsto u^{\otimes2}\in (T_{\theta_0} \Theta)^{\otimes 2}  .
\]
For the identity polynomial, \eqref{eq:centered-OU-Poisson-equation} gives
\[
  \calP u=2u,
\]
and hence, by the definition of $\calP_n$,
\begin{equation}
  \calP_nu=2(u-\zeta_{n,0}).
  \label{eq:OU-Poisson-identity}
\end{equation}
Together with \eqref{eq:first-OU-correction}, this gives
\[
\begin{aligned}
  (\frakL_n^{[1]}\calP_nu)(u)
  ={}&
  \Upsilon_{n,0}[u]
  +a_nV_0
  \\
  &+
  a_n(\nabla-D^{(-1)})_0
  \left[u,\zeta_{n,0}-\frac12u\right].
\end{aligned}
\]
Since
\[
  \E_N[N_n]=\zeta_{n,0},
  \qquad
  \E_N[N_n^{\otimes2}]
  =
  \zeta_{n,0}^{\otimes2}+\tau g_0^{-1},
\]
we obtain
\[
\begin{aligned}
  \E_N\!\left[
    (\frakL_n^{[1]}\calP_nu)(N_n)
  \right]
  ={}&
  \Upsilon_{n,0}[\zeta_{n,0}]
  +a_nV_0
+  \frac{a_n}{2}
    (\nabla-D^{(-1)})_0
    [\zeta_{n,0}^{\otimes2}-\tau g_0^{-1}].
\end{aligned}
\]
Proposition~\ref{prop:polynomial-Poisson-expansion} therefore yields
\begin{equation}
\begin{aligned}
  \int_\Theta \bar u_n\,\dd\mu_n
  ={}&
  \zeta_{n,0}
  +\Upsilon_{n,0}[\zeta_{n,0}]
  +a_nV_0
  \\
  &+
  \frac{a_n}{2}
    (\nabla-D^{(-1)})_0
    [\zeta_{n,0}^{\otimes2}-\tau g_0^{-1}]
  +o_{\Linftyminus(\theta_0)}(a_n).
\end{aligned}
\label{eq:OU-first-moment-expansion}
\end{equation}

For the quadratic polynomial, the coefficient-order preservation of
$\calP_n$ shows that $\calP_n(u^{\otimes2})$ has degree $2$ and $O_{\Linftyminus(\theta_0)}(1)$ coefficients. By
\eqref{eq:first-OU-correction} and the bounds \( \zeta_{n,0}=O_{\Linftyminus(\theta_0)}(1),
  \Upsilon_{n,0}=o_{\Linftyminus(\theta_0)}(a_n^{1/2})
\) from \eqref{eq:zeta,Z-order}, the coefficients of
$\frakL_n^{[1]}\calP_n(u^{\otimes2})$ are
$o_{\Linftyminus(\theta_0)}(a_n^{1/2})$. Since $N_n=\zeta_{n,0}+N$,
\begin{align*}
    \E_N[ N_n^{\otimes p} ] = O_{\Linftyminus(\theta_0)}(1)
\end{align*}
for all $ p \in \mathbb{N}$. Thus
\[
  \E_N\!\left[
    (\frakL_n^{[1]}\calP_n(u^{\otimes2}))(N_n)
  \right]
  =
  o_{\Linftyminus(\theta_0)}(1).
\]
Since
\[
  \E_N[N_n^{\otimes2}]
  =
  \zeta_{n,0}^{\otimes2}+\tau g_0^{-1},
\]
Proposition~\ref{prop:polynomial-Poisson-expansion} also gives
\begin{equation}
  \int_\Theta \bar u_n^{\otimes2}\,\dd\mu_n
  =
  \zeta_{n,0}^{\otimes2}
  +\tau g_0^{-1}
  +o_{\Linftyminus(\theta_0)}(1).
  \label{eq:OU-second-moment-expansion}
\end{equation}

\begin{proof}[Proof of \eqref{eq:main-expansion} in Theorem~\ref{thm:main-expansion}]
Fix $f\in C^\infty(\Theta;\R^{\mathrm q})$. Throughout the proof,
$C>0$ denotes a deterministic constant independent of $n$, whose value
may change from line to line.

Since $\bar u_n=a_n^{-1}\log_{\theta_0}$ on $O_n$, Taylor's formula in
the $\nabla$-normal coordinates at $\theta_0$ gives, for
$\theta\in O_n$,
\begin{equation}
\begin{aligned}
  f(\theta)
  ={}&
  f(\theta_0)
  +a_n\,(\dd f)_0[\bar u_n(\theta)]
  +\frac{a_n^2}{2}
    (\nabla\dd f)_0[\bar u_n(\theta)^{\otimes2}]
  +r_n(\theta),
\end{aligned}
\label{eq:f-normal-Taylor}
\end{equation}
where
\begin{equation}
  |r_n(\theta)|
  \le
  Ca_n^3|\bar u_n(\theta)|^3,
  \qquad
  \theta\in O_n.
  \label{eq:f-normal-Taylor-remainder}
\end{equation}
By \eqref{eq:globalized-OU-moment-bound},
\begin{equation}
  \int_{O_n}|r_n(\theta)|\,\mu_n(\dd\theta)
  =
  O_{\Linftyminus(\theta_0)}(a_n^3).
  \label{eq:f-Taylor-remainder-integral}
\end{equation}
Therefore,
\begin{align}
  \mu_n(f)
  &=
  \int_{O_n} f(\theta)\,\mu_n(\dd\theta)
  +o_{\Linftyminus(\theta_0)}(a_n^2)
  \notag\\[-1mm]
  &\hspace{28mm}
  \text{(by Assumption~\textup{[A2]}$(\mu_n)$ and }\|f\|_\infty<\infty\text{)}
  \notag\\
  &=
  \int_{O_n}
  \Bigg\{
    f(\theta_0)
    +a_n\,(\dd f)_0[\bar u_n(\theta)]
    +\frac{a_n^2}{2}
      (\nabla\dd f)_0[\bar u_n(\theta)^{\otimes2}]
  \Bigg\}\mu_n(\dd\theta)
  +o_{\Linftyminus(\theta_0)}(a_n^2)
  \notag\\[-1mm]
  &\hspace{28mm}
  \text{(by \eqref{eq:f-normal-Taylor} and
  \eqref{eq:f-Taylor-remainder-integral})}
  \notag\\
  &=
  \int_{\Theta}
  \Bigg\{
    f(\theta_0)
    +a_n\,(\dd f)_0[\bar u_n(\theta)]
    +\frac{a_n^2}{2}
      (\nabla\dd f)_0[\bar u_n(\theta)^{\otimes2}]
  \Bigg\}\mu_n(\dd\theta)
  +o_{\Linftyminus(\theta_0)}(a_n^2)
  \notag\\[-1mm]
  &\hspace{28mm}
  \text{(by \eqref{eq:polynomial-tail-bound})}
  \notag\\
  &=
  f(\theta_0)
  +a_n\,(\dd f)_0\!\left[
    \int_\Theta \bar u_n\,\dd\mu_n
  \right]
  +\frac{a_n^2}{2}
  (\nabla\dd f)_0\!\left[
    \int_\Theta \bar u_n^{\otimes2}\,\dd\mu_n
  \right]
  +o_{\Linftyminus(\theta_0)}(a_n^2).
  \label{eq:f-integral-moment-reduction}
\end{align}

Substituting \eqref{eq:OU-first-moment-expansion} and
\eqref{eq:OU-second-moment-expansion} into
\eqref{eq:f-integral-moment-reduction} yields
\begin{equation}
\begin{aligned}
  \mu_n(f)
  ={}&
  f(\theta_0)
  +a_n\,(\dd f)_0[
    \zeta_{n,0}+\Upsilon_{n,0}[\zeta_{n,0}]
  ]
  \\
  &+
  a_n^2\Bigg\{
    (\dd f)_0[V_0]
    +\frac12
      (\dd f)_0\!\left[
        (\nabla-D^{(-1)})_0
        [\zeta_{n,0}^{\otimes2}-\tau g_0^{-1}]
      \right]
  \\
  &\hspace{28mm}
    +\frac12
      (\nabla\dd f)_0
      [\zeta_{n,0}^{\otimes2}+\tau g_0^{-1}]
  \Bigg\}
  +o_{\Linftyminus(\theta_0)}(a_n^2).
\end{aligned}
\label{eq:f-expansion-before-connection-reduction}
\end{equation}

By the dual connection formula, for every
$S\in T_{\theta_0}\Theta^{\otimes2}$,
\begin{equation}
  (D^{(-1)}\dd f)_0[S]
  =
  (\nabla\dd f)_0[S]
  +
  (\dd f)_0\!\left[
    (\nabla-D^{(-1)})_0[S]
  \right].
  \label{eq:Dminus1-df-normal-relation}
\end{equation}
Using \eqref{eq:Dminus1-df-normal-relation}, the second-order terms in
\eqref{eq:f-expansion-before-connection-reduction} other than
$(\dd f)_0[V_0]$ become
\[
\begin{aligned}
  &\frac12
    (\dd f)_0\!\left[
      (\nabla-D^{(-1)})_0
      [\zeta_{n,0}^{\otimes2}-\tau g_0^{-1}]
    \right]
  +
  \frac12
    (\nabla\dd f)_0
    [\zeta_{n,0}^{\otimes2}+\tau g_0^{-1}]
  \\
  &\qquad=
  \tau(\Delta f)(\theta_0)
  +
  \frac12
    (D^{(-1)}\dd f)_0
    [\zeta_{n,0}^{\otimes2}-\tau g_0^{-1}]
  \\
  &\qquad=
  \left(
    \tau\Delta+\frac{1-\tau}{2}\Delta^{(-1)}
  \right)f(\theta_0)
  +
  \frac12
    (D^{(-1)}\dd f)_0
    [\zeta_{n,0}^{\otimes2}-g_0^{-1}].
\end{aligned}
\]
Substituting this identity into
\eqref{eq:f-expansion-before-connection-reduction} gives
\[
  \mu_n(f)
  =
  (\mathcal M_nf)(\theta_0)
  +o_{\Linftyminus(\theta_0)}(a_n^2),
\]
which is \eqref{eq:main-expansion}.
\end{proof}

\subsection{Expansion of the decision rule}
\label{subsec:proof-decision-expansion}

Choose a coordinate neighborhood $G$ of $m(\theta_0)$ such that its
coordinate image is a bounded convex open subset of the corresponding
Euclidean space. Throughout this subsection~\ref{subsec:proof-decision-expansion}, we identify $G$ with
this coordinate image through the chosen chart.

Fix a basis $(e_1,\ldots,e_{\mathrm p})$ of $T_{\theta_0}\Theta$ and use
the corresponding normal coordinates $(x^1,\ldots,x^{\mathrm p})$ on
$B_\Theta(\theta_0,r_0)$, defined by
\[
  \log_{\theta_0}\theta=x^i(\theta)e_i.
\]
We use indices $a,b,c,\ldots$ for the coordinates on $G$ and
$i,j,k,\ldots$ for the normal coordinates on $B_\Theta(\theta_0,r_0)$.
Denote the corresponding coordinate vector fields by
\[
  \partial_a:=\frac{\partial}{\partial z^a},
  \qquad
  \partial_i:=\frac{\partial}{\partial x^i},
\]
respectively.

Choose $r_1\in(0,a_{n_0}^{\kappa})$ sufficiently small and set
\[
  O:=B_\Theta(\theta_0,r_1).
\]
Decreasing $r_1$ and $G$ if necessary, we may assume that \(  \overline{O}\times\overline G
\) is contained in the neighborhood of $\operatorname{Graph}(m)$ on which $W$ is smooth, and \(m(\overline{O}) \subset G\). We may further assume that there exists $c>0$ such that
\begin{align}
  v^av^b\,\partial_a\partial_bW(\theta,z)
  &\ge
  c|v|^2,
  \qquad
  \theta\in\overline{O},\quad
  z\in\overline G,
  \label{eq:decision-local-Hessian-lower-bound}
  \\
  W(\theta,z)-W(\theta,m(\theta_0))
  &\ge
  c,
  \qquad
  \theta\in O,\quad
  z\in Z\setminus G.
  \label{eq:decision-local-gap}
\end{align}
Indeed, the first inequality follows from the positive definiteness of
the Hessian of $W(\theta_0,\cdot)$ at $m(\theta_0)$ and continuity.
For the second, since $m(\theta_0)$ is the unique minimizer of
$W(\theta_0,\cdot)$ and $Z\setminus G$ is compact,
\[
  \inf_{z\in Z\setminus G}
  \left\{
    W(\theta_0,z)-W(\theta_0,m(\theta_0))
  \right\}
  >0.
\]
Decreasing $c$ if necessary, we may therefore assume that
\begin{equation}
  W(\theta_0,z)-W(\theta_0,m(\theta_0))
  \ge
  2c,
  \qquad
  z\in Z\setminus G.
  \label{eq:decision-gap-at-theta0}
\end{equation}
Now set
\[
  F(\theta,z)
  :=
  W(\theta,z)-W(\theta,m(\theta_0)).
\]
Since $F$ is continuous and $Z\setminus G$ is compact,
\[
  \sup_{z\in Z\setminus G}
  |F(\theta,z)-F(\theta_0,z)|
  \longrightarrow
  0
  \qquad
  \text{as }\theta\to\theta_0.
\]
Hence, after decreasing $r_1$ if necessary,
\begin{equation}
  \sup_{z\in Z\setminus G}
  |F(\theta,z)-F(\theta_0,z)|
  \le
  c,
  \qquad
  \theta\in O.
  \label{eq:decision-gap-uniform-continuity}
\end{equation}
By \eqref{eq:decision-gap-at-theta0} and
\eqref{eq:decision-gap-uniform-continuity},
\[
  F(\theta,z)\ge c,
  \qquad
  \theta\in O,\quad z\in Z\setminus G,
\]
which is \eqref{eq:decision-local-gap}.

For a Borel set $A\subset\Theta$, define
\begin{equation}
  \widetilde\mu_n(A)
  :=
  \mu_n(A\cap O)
  +
  \mu_n(O^c)\delta_{\theta_0}(A).
  \label{eq:decision-localized-measure}
\end{equation}
Then $\widetilde\mu_n$ is a probability measure. Define
\begin{equation}
  \widetilde Q_n(z)
  :=
  \int_\Theta W(\theta,z)\,\widetilde\mu_n(\dd\theta).
  \label{eq:decision-localized-objective}
\end{equation}

\begin{lemma}
\label{lem:decision-localization}
For $O$ and $G$ chosen as above,
$\widetilde Q_n$ has, for all sufficiently large $n$, a unique minimizer
$\widetilde z_n\in G$. Moreover, for every $L>0$, every fixed
$\gamma\in C^\infty(Z;\R^{\mathrm q})$, every fixed
$f\in C(\Theta;\R^{\mathrm q})$, and every random polynomial $q_n$ on
$T_{\theta_0}\Theta$, with values in a fixed finite-dimensional vector
space, whose degree is uniformly bounded and whose coefficients are
$O_{\Linftyminus(\theta_0)}(1)$,
\begin{align}
  \mu_n\!\left(O^c\right)
  &=
  O_{\Linftyminus(\theta_0)}(a_n^L),
  \label{eq:decision-fixed-ball-tail}
  \\
  \gamma(\widehat z_n)-\gamma(\widetilde z_n)
  &=
  O_{\Linftyminus(\theta_0)}(a_n^L),
  \label{eq:decision-gamma-localization}
  \\
  \mu_n(f)-\widetilde\mu_n(f)
  &=
  O_{\Linftyminus(\theta_0)}(a_n^L),
  \label{eq:decision-integral-localization}
  \\
  \widetilde\mu_n\!\left(q_n(\bar u_n)\right)
  -\mu_n\!\left(q_n(\bar u_n)\right)
  &=
  O_{\Linftyminus(\theta_0)}(a_n^L).
  \label{eq:decision-polynomial-localization}
\end{align}
\end{lemma}

\begin{proof}
For all sufficiently large $n$, $a_n^\kappa<r_1$, so
\[
  O^c
  \subset
  B_\Theta(\theta_0,a_n^\kappa)^c.
\]
For every $p<\infty$ and $L>0$, Assumption~\textup{[A2]}$(\mu_n)$, applied with
the power $pL$, gives
\[
\begin{aligned}
  \left\|
    \mu_n\!\left(O^c\right)
  \right\|_{L^p(P_{n,\theta_0})}^p
  &\le
  \E_{n,\theta_0}\!\left[
    \mu_n\!\left(B_\Theta(\theta_0,a_n^\kappa)^c\right)
  \right]
  \\
  &=O(a_n^{pL}),
\end{aligned}
\]
which proves \eqref{eq:decision-fixed-ball-tail}.

Integrating \eqref{eq:decision-local-gap} with respect to
$\widetilde\mu_n$ gives
\begin{equation}
  \widetilde Q_n(z)-\widetilde Q_n(m(\theta_0))
  \ge
  c,
  \qquad
  z\in Z\setminus G.
  \label{eq:decision-localized-minimum-gap}
\end{equation}
Similarly, \eqref{eq:decision-local-Hessian-lower-bound} gives
\begin{equation}
  v^av^b\,\partial_a\partial_b\widetilde Q_n(z)
  \ge
  c|v|^2,
  \qquad
  z\in G.
  \label{eq:decision-localized-Hessian-lower-bound}
\end{equation}
Since $Z$ is compact, $\widetilde Q_n$ has a minimizer. By
\eqref{eq:decision-localized-minimum-gap}, every minimizer belongs to
$G$, while \eqref{eq:decision-localized-Hessian-lower-bound} implies
that $\widetilde Q_n$ is strictly convex on $G$. Hence
$\widetilde Q_n$ has a unique minimizer $\widetilde z_n\in G$.

By the definition of $\widetilde\mu_n$ in
\eqref{eq:decision-localized-measure}, for every $z\in Z$,
\[
  Q_n(z)-\widetilde Q_n(z)
  =
  \int_{O^c}
  \left\{
    W(\theta,z)-W(\theta_0,z)
  \right\}\mu_n(\dd\theta).
\]
Hence, with
\[
  \varepsilon_n
  :=
  \sup_{z\in Z}|Q_n(z)-\widetilde Q_n(z)|,
\]
we have
\begin{equation}
  \varepsilon_n
  \le
  2\|W\|_\infty
  \mu_n\!\left(O^c\right)
  =
  O_{\Linftyminus(\theta_0)}(a_n^L)
  \label{eq:decision-objective-localization}
\end{equation}
for every $L>0$.

We claim that $\widehat z_n\in G$ on
$\{\varepsilon_n<c/2\}$. Indeed, if $\widehat z_n\notin G$, then
\eqref{eq:decision-localized-minimum-gap} and the definition of
$\varepsilon_n$ give
\[
\begin{aligned}
  Q_n(\widehat z_n)
  &\ge
  \widetilde Q_n(\widehat z_n)-\varepsilon_n
  \\
  &\ge
  \widetilde Q_n(m(\theta_0))+c-\varepsilon_n,
\end{aligned}
\]
whereas
\[
  Q_n(m(\theta_0))
  \le
  \widetilde Q_n(m(\theta_0))+\varepsilon_n.
\]
Hence
\[
  Q_n(\widehat z_n)-Q_n(m(\theta_0))
  \ge
  c-2\varepsilon_n
  >0
\]
on $\{\varepsilon_n<c/2\}$, contradicting the minimizing property of
$\widehat z_n$. Thus $\widehat z_n\in G$ on this event.

Since then both $\widehat z_n$ and
$\widetilde z_n$ belong to the convex set $G$, the line segment joining
them is contained in $G$. Moreover,
$\partial_a\widetilde Q_n(\widetilde z_n)=0$ because
$\widetilde z_n$ is the minimizer of $\widetilde Q_n$. Thus Taylor's formula and \eqref{eq:decision-localized-Hessian-lower-bound} give
\begin{equation}
\begin{aligned}
  &\widetilde Q_n(\widehat z_n)-\widetilde Q_n(\widetilde z_n)
  \\
  &\quad=
  \int_0^1(1-t)
  (\widehat z_n-\widetilde z_n)^a
  (\widehat z_n-\widetilde z_n)^b
  \partial_a\partial_b\widetilde Q_n
  \bigl(\widetilde z_n+t(\widehat z_n-\widetilde z_n)\bigr)\,\dd t
  \\
  &\quad\ge
  \frac c2|\widehat z_n-\widetilde z_n|^2.
\end{aligned}
\label{eq:decision-objective-lower-bound}
\end{equation}
On the other hand, since $\widehat z_n$ minimizes $Q_n$ and
$|Q_n-\widetilde Q_n|\le\varepsilon_n$ on $Z$,
\[
\begin{aligned}
  \widetilde Q_n(\widehat z_n)
  &\le
  Q_n(\widehat z_n)+\varepsilon_n
  \\
  &\le
  Q_n(\widetilde z_n)+\varepsilon_n
  \\
  &\le
  \widetilde Q_n(\widetilde z_n)+2\varepsilon_n.
\end{aligned}
\]
Hence
\begin{equation}
  \widetilde Q_n(\widehat z_n)-\widetilde Q_n(\widetilde z_n)
  \le
  2\varepsilon_n.
  \label{eq:decision-objective-upper-bound}
\end{equation}
By \eqref{eq:decision-objective-lower-bound} and
\eqref{eq:decision-objective-upper-bound},
\begin{equation}
  |\widehat z_n-\widetilde z_n|^2
  \le
  \frac4c\,\varepsilon_n
  \qquad
  \text{on }\{\varepsilon_n<c/2\}.
  \label{eq:decision-minimizer-localization}
\end{equation}
Let $C_\gamma$ be a Lipschitz constant of $\gamma$ on $\overline G$.
Since $\widetilde z_n\in G$ and $\widehat z_n\in G$ on
$\{\varepsilon_n<c/2\}$,
\begin{align}
      |\gamma(\widehat z_n)-\gamma(\widetilde z_n)|
  \le{}&
  C_\gamma|\widehat z_n-\widetilde z_n|
  \mathbf 1_{\{\varepsilon_n<c/2\}}
 +
  2\|\gamma\|_\infty
  \mathbf 1_{\{\varepsilon_n\ge c/2\}} \\  \le{}&  \frac{2C_\gamma}{c^{1/2}}\,\varepsilon_n^{1/2} +
  2\|\gamma\|_\infty
  \mathbf 1_{\{\varepsilon_n\ge c/2\}},
\end{align}
where the second inequality follows from
\eqref{eq:decision-minimizer-localization}.
Together with \eqref{eq:decision-objective-localization} and Markov's
inequality, this proves \eqref{eq:decision-gamma-localization}.

Finally, for fixed $f\in C(\Theta;\R^{\mathrm q})$,
\[
  \mu_n(f)-\widetilde\mu_n(f)
  =
  \int_{O^c}
  \{f(\theta)-f(\theta_0)\}\,\mu_n(\dd\theta),
\]
and hence
\[
  |\mu_n(f)-\widetilde\mu_n(f)|
  \le
  2\|f\|_\infty
  \mu_n\!\left(O^c\right).
\]
Equation~\eqref{eq:decision-integral-localization} follows from
\eqref{eq:decision-fixed-ball-tail}.

Finally, by \eqref{eq:decision-localized-measure},
\[
\begin{aligned}
  &\widetilde\mu_n\!\left(q_n(\bar u_n)\right)
  -\mu_n\!\left(q_n(\bar u_n)\right)
  \\
  &\quad=
  -\int_{O^c}q_n(\bar u_n)\,\dd\mu_n
  +q_n(0)\mu_n(O^c).
\end{aligned}
\]
For all sufficiently large $n$, $O_n\subset O$, and hence
$O^c\subset O_n^c$. Thus \eqref{eq:polynomial-tail-bound} and \eqref{eq:decision-fixed-ball-tail} give
\eqref{eq:decision-polynomial-localization}.
\end{proof}

For later use, set
\begin{equation}
  \widetilde u_n
  :=
  \widetilde\mu_n(\bar u_n),
  \qquad
  \widetilde\theta_n
  :=
  \exp_{\theta_0}(a_n\widetilde u_n).
  \label{eq:decision-local-center}
\end{equation}
Since $r_1<a_{n_0}^{\kappa}$, we have
$\bar x=\log_{\theta_0}$ on $O$. Hence
\[
  a_n\widetilde u_n
  =
  \widetilde\mu_n(\bar x)
  \in
  B_{T_{\theta_0}\Theta}(0,r_1),
\]
because the ball is convex. Thus
\[
  \widetilde\theta_n\in O,
  \qquad
  m(\widetilde\theta_n)\in G.
\]

\begin{proof}[Proof of \eqref{eq:loss-explicit-expansion} in Theorem~\ref{thm:main-expansion}]
Fix $\gamma\in C^\infty(Z;\R^{\mathrm q})$. Throughout the proof,
$C>0$ denotes a deterministic constant independent of $n$, whose value
may change from line to line. By
\eqref{eq:main-expansion}, applied to $f=\gamma\circ m$, it suffices to
show that
\begin{equation}
  \gamma(\widehat z_n)
  =
  \mu_n(\gamma\circ m)
  -\frac{\tau a_n^2}{2}(\Delta^W\gamma)(\theta_0)
  +o_{\Linftyminus(\theta_0)}(a_n^2).
  \label{eq:decision-integral-expansion}
\end{equation}
By Lemma~\ref{lem:decision-localization}, it is enough to prove
\begin{equation}
  \gamma(\widetilde z_n)
  =
  \widetilde\mu_n(\gamma\circ m)
  -\frac{\tau a_n^2}{2}(\Delta^W\gamma)(\theta_0)
  +o_{\Linftyminus(\theta_0)}(a_n^2).
  \label{eq:decision-localized-expansion}
\end{equation}

We first record the centered moments under $\widetilde\mu_n$.
Applying \eqref{eq:decision-polynomial-localization} to
$u\mapsto u$ and $u\mapsto u^{\otimes2}$, and then using
\eqref{eq:OU-first-moment-expansion} and
\eqref{eq:OU-second-moment-expansion}, gives
\[
\widetilde u_n =   \widetilde\mu_n(\bar u_n)
  =
  \zeta_{n,0}
  +\Upsilon_{n,0}[\zeta_{n,0}]
  +O_{\Linftyminus(\theta_0)}(a_n)
  +o_{\Linftyminus(\theta_0)}(a_n),
\]
and
\[
  \widetilde\mu_n(\bar u_n^{\otimes2})
  =
  \zeta_{n,0}^{\otimes2}
  +\tau g_0^{-1}
  +o_{\Linftyminus(\theta_0)}(1).
\]
Since
\[
  \zeta_{n,0}=O_{\Linftyminus(\theta_0)}(1),
  \qquad
  \Upsilon_{n,0}=o_{\Linftyminus(\theta_0)}(a_n^{1/2})
\]
by \eqref{eq:zeta,Z-order}, we obtain
\begin{align}
  \widetilde u_n
  =
  \widetilde\mu_n(\bar u_n)
  &=
  \zeta_{n,0}
  +o_{\Linftyminus(\theta_0)}(a_n^{1/2}),
  \label{eq:decision-centered-first-moment}
  \\
  \widetilde\mu_n\!\left[
    (\bar u_n-\widetilde u_n)^{\otimes2}
  \right]
  &=
  \tau g_0^{-1}
  +o_{\Linftyminus(\theta_0)}(1).
  \label{eq:decision-centered-second-moment}
\end{align}
By definition,
\begin{equation}
  \widetilde\mu_n(\bar u_n-\widetilde u_n)=0.
  \label{eq:decision-centered-first-moment-zero}
\end{equation}
Finally, \eqref{eq:globalized-OU-moment-bound} and
\eqref{eq:decision-polynomial-localization}, applied to
$(1+|u|^2)^2$, together with
\eqref{eq:decision-centered-first-moment}, give
\begin{equation}
  \widetilde\mu_n\!\left(
    |\bar u_n-\widetilde u_n|^3
  \right)
  =
  O_{\Linftyminus(\theta_0)}(1).
  \label{eq:decision-centered-third-moment}
\end{equation}
In particular,
\begin{align}
      x(\widetilde\theta_n)-x(\theta_0)  =  x(\widetilde \theta_n)
  =
  a_n\widetilde u_n
  =
  O_{\Linftyminus(\theta_0)}(a_n). \label{eq:difference-theta}
\end{align}

Since $\widetilde\mu_n$ is supported on $\overline O$ and $W$ is smooth
on a neighborhood of $\overline O\times\overline G$, all derivatives of
$W$ used below are uniformly bounded there. Hence differentiation under
the $\widetilde\mu_n$-integral is justified by dominated convergence; in
particular, for $z\in G$,
\[
  \partial_a\widetilde Q_n(z)
  =
  \int_\Theta \partial_aW(\theta,z)\,\widetilde\mu_n(\dd\theta),
  \qquad
  \partial_a\partial_b\widetilde Q_n(z)
  =
  \int_\Theta \partial_a\partial_bW(\theta,z)\,
  \widetilde\mu_n(\dd\theta).
\]
We use this interchange without further comment below.

By \eqref{eq:loss-stationarity}, the Taylor expansion of
$\theta\mapsto\partial_aW(\theta,m(\widetilde\theta_n))$ around
$\widetilde\theta_n$ gives, for $\theta\in O$,
\begin{align}
  \partial_aW(\theta,m(\widetilde\theta_n))
  ={}&
  a_n
  W(\partial_i\mid\partial_a)(\widetilde\theta_n)
  (\bar u_n^i-\widetilde u_n^i)
  \notag\\
  &+
  \frac{a_n^2}{2}
  W(\partial_i\partial_j\mid\partial_a)(\widetilde\theta_n)
  (\bar u_n^i-\widetilde u_n^i)
  (\bar u_n^j-\widetilde u_n^j)
  +r_{n,a}(\theta),
  \label{eq:decision-score-pointwise-Taylor}
\end{align}
where
\[
  |r_{n,a}(\theta)|
  \le
  Ca_n^3|\bar u_n-\widetilde u_n|^3.
\]
By
\eqref{eq:loss-derivative-definition},
\begin{equation}
\begin{aligned}
  W(\partial_i\partial_j\mid\partial_a)(\theta_0)
  &=
  -\braket{D^W_{\partial_i}\partial_j,\partial_a}_{g_W}(\theta_0)
  \\
  &=
  -(g_W)_{ab}(\theta_0)
  \bigl(D^W_{\partial_i}\partial_j\bigr)^b(\theta_0).
\end{aligned}
\label{eq:decision-loss-jet-coordinate}
\end{equation}
Here $(g_W)_{ab}(\theta_0)$ are the coordinate components of the Hessian
of $W(\theta_0,\cdot)$ at $m(\theta_0)$. 

Integrating \eqref{eq:decision-score-pointwise-Taylor} and using
\eqref{eq:decision-centered-second-moment}--%
\eqref{eq:decision-centered-third-moment} gives
\begin{align}
  \partial_a\widetilde Q_n(m(\widetilde\theta_n))
  ={}&
  \frac{\tau a_n^2}{2}
  g_0^{ij}
  W(\partial_i\partial_j\mid\partial_a)(\widetilde\theta_n)
  +o_{\Linftyminus(\theta_0)}(a_n^2)
  \notag\\
  ={}&
  \frac{\tau a_n^2}{2}
  g_0^{ij}
  W(\partial_i\partial_j\mid\partial_a)(\theta_0)
  +o_{\Linftyminus(\theta_0)}(a_n^2)
  \notag\\
  ={}&
  -\frac{\tau a_n^2}{2}
  g_0^{ij}(g_W)_{ab}(\theta_0)
  \bigl(D^W_{\partial_i}\partial_j\bigr)^b(\theta_0)
  +o_{\Linftyminus(\theta_0)}(a_n^2).
  \label{eq:decision-score-expansion}
\end{align}
Here the second equality follows from \eqref{eq:difference-theta} and the
smoothness of $W(\partial_i\partial_j\mid\partial_a)$, while the last
equality is \eqref{eq:decision-loss-jet-coordinate}. Also, $g_0^{ij}$
denotes the $(i,j)$-component of $g_0^{-1}$.
The same argument, applied to
$\partial_a\partial_bW(\theta,m(\widetilde\theta_n))$, gives
\begin{equation}
  \partial_a\partial_b
  \widetilde Q_n(m(\widetilde\theta_n))
  =
  (g_W)_{ab}(\theta_0)
  +o_{\Linftyminus(\theta_0)}(1).
  \label{eq:decision-Hessian-expansion}
\end{equation}

Using $\partial_a\widetilde Q_n(\widetilde z_n)=0$, Taylor's formula and \eqref{eq:decision-localized-Hessian-lower-bound} give
\[
\begin{aligned}
  &-\partial_a\widetilde Q_n(m(\widetilde\theta_n))
  \{\widetilde z_n-m(\widetilde\theta_n)\}^a
  \\
  &\quad=
  \int_0^1
  \{\widetilde z_n-m(\widetilde\theta_n)\}^a
  \{\widetilde z_n-m(\widetilde\theta_n)\}^b
  \partial_a\partial_b\widetilde Q_n
  \bigl(
    m(\widetilde\theta_n)
    +t\{\widetilde z_n-m(\widetilde\theta_n)\}
  \bigr)\,\dd t
  \\
  &\quad\ge
  c|\widetilde z_n-m(\widetilde\theta_n)|^2.
\end{aligned}
\]
Hence,
\begin{equation}
  \left|
    \partial_a\widetilde Q_n(m(\widetilde\theta_n))
    \{\widetilde z_n-m(\widetilde\theta_n)\}^a
  \right|
  \ge
  c|\widetilde z_n-m(\widetilde\theta_n)|^2.
  \label{eq:decision-score-displacement-bound}
\end{equation}
Together with \eqref{eq:decision-score-expansion}, this yields
\begin{equation}
  \widetilde z_n-m(\widetilde\theta_n)
  =
  O_{\Linftyminus(\theta_0)}(a_n^2).
  \label{eq:decision-local-minimizer-order}
\end{equation}

Using \eqref{eq:decision-local-minimizer-order}, the first-order condition
can now be expanded as
\begin{align}
  0
  &=
  \partial_a\widetilde Q_n(\widetilde z_n)
  \notag\\[-1mm]
  &\hspace{28mm}
  \text{(by the definition of $\widetilde z_n$)}
  \notag\\
  &=
  \partial_a\widetilde Q_n(m(\widetilde\theta_n))
  +
  \partial_a\partial_b
  \widetilde Q_n(m(\widetilde\theta_n))
  \{\widetilde z_n-m(\widetilde\theta_n)\}^b
  +
  o_{\Linftyminus(\theta_0)}(a_n^2)
  \notag\\[-1mm]
  &\hspace{28mm}
  \text{(by Taylor's formula)}
  \notag\\
  &=
  -\frac{\tau a_n^2}{2}
  g_0^{ij}(g_W)_{ab}(\theta_0)
  \bigl(D^W_{\partial_i}\partial_j\bigr)^b(\theta_0)
  +
  (g_W)_{ab}(\theta_0)
  \{\widetilde z_n-m(\widetilde\theta_n)\}^b
  +
  o_{\Linftyminus(\theta_0)}(a_n^2)
  \label{eq:decision-first-order-linearization}\\[-1mm]
  &\hspace{28mm}
  \text{(by \eqref{eq:decision-score-expansion} and
  \eqref{eq:decision-Hessian-expansion}).}
  \notag
\end{align}
Since $(g_W)_{ab}(\theta_0)$ is invertible,
\eqref{eq:decision-first-order-linearization} is equivalent to
\begin{equation}
  \{\widetilde z_n-m(\widetilde\theta_n)\}^a
  =
  \frac{\tau a_n^2}{2}
  g_0^{ij}
  \bigl(D^W_{\partial_i}\partial_j\bigr)^a(\theta_0)
  +
  o_{\Linftyminus(\theta_0)}(a_n^2).
  \label{eq:decision-local-minimizer-expansion}
\end{equation}
Consequently, by Taylor's formula and
\eqref{eq:decision-local-minimizer-expansion},
\begin{align}
  \gamma(\widetilde z_n)
  -\gamma(m(\widetilde\theta_n))
  &=
  (\dd\gamma\circ m)(\theta_0)
  \bigl[
    \widetilde z_n-m(\widetilde\theta_n)
  \bigr]
  +o_{\Linftyminus(\theta_0)}(a_n^2)
  \notag\\
  &=
  \frac{\tau a_n^2}{2}
  g_0^{ij}
  (\dd\gamma\circ m)
  [D^W_{\partial_i}\partial_j](\theta_0)
  +
  o_{\Linftyminus(\theta_0)}(a_n^2).
  \label{eq:decision-gamma-displacement}
\end{align}
Here
\[
  (\dd\gamma\circ m)(\theta_0)
  =
  \dd\gamma_{m(\theta_0)}
  \in
  T_{m(\theta_0)}^*Z\otimes\R^{\mathrm q},
\]
which should not be confused with
$\dd(\gamma\circ m)_{\theta_0}$.

Set $f:=\gamma\circ m$. Taylor expansion of $f$ around
$\widetilde\theta_n$ gives, for $\theta\in O$,
\[
\begin{aligned}
  f(\theta)-f(\widetilde\theta_n)
  ={}&
  a_n\partial_i f(\widetilde\theta_n)
  (\bar u_n^i-\widetilde u_n^i)
  \\
  &+
  \frac{a_n^2}{2}
  \partial_i\partial_jf(\widetilde\theta_n)
  (\bar u_n^i-\widetilde u_n^i)
  (\bar u_n^j-\widetilde u_n^j)
  +r_n^f(\theta),
\end{aligned}
\]
where
\[
  |r_n^f(\theta)|
  \le
  Ca_n^3|\bar u_n-\widetilde u_n|^3.
\]
Integrating this identity and using
\eqref{eq:decision-centered-second-moment}--%
\eqref{eq:decision-centered-third-moment}, we obtain
\begin{align}
  \widetilde\mu_n(f)-f(\widetilde\theta_n)
  ={}&
  \frac{\tau a_n^2}{2}
  g_0^{ij}
  \partial_i\partial_jf(\widetilde\theta_n)
  +o_{\Linftyminus(\theta_0)}(a_n^2)
  \notag\\
  ={}&
  \frac{\tau a_n^2}{2}
  g_0^{ij}
  \partial_i\partial_jf(\theta_0)
  +o_{\Linftyminus(\theta_0)}(a_n^2).
  \label{eq:decision-target-integral-expansion}
\end{align}
Here the second equality follows from
\eqref{eq:difference-theta} and the smoothness of $f$.

Now
\[
\begin{aligned}
  \gamma(\widetilde z_n)-\widetilde\mu_n(\gamma\circ m)
  ={}&
  \bigl\{
    \gamma(\widetilde z_n)-\gamma(m(\widetilde\theta_n))
  \bigr\}
  \\
  &-
  \bigl\{
    \widetilde\mu_n(\gamma\circ m)
    -\gamma(m(\widetilde\theta_n))
  \bigr\}.
\end{aligned}
\]
Therefore, by \eqref{eq:decision-gamma-displacement} and
\eqref{eq:decision-target-integral-expansion},
\[
\begin{aligned}
  \gamma(\widetilde z_n)-\widetilde\mu_n(\gamma\circ m)
  =
  -\frac{\tau a_n^2}{2}g_0^{ij}
  \Bigl\{
    \partial_i\partial_j(\gamma\circ m)
    -
    (\dd\gamma\circ m)
    [D^W_{\partial_i}\partial_j]
  \Bigr\}(\theta_0)
  +
  o_{\Linftyminus(\theta_0)}(a_n^2).
\end{aligned}
\]
By the definition of the dual O-derivative and the fact that the
principal homomorphism of $D^W$ is $\dd m$,
\[
\begin{aligned}
  \bigl(D^W_{\partial_i}(\dd\gamma\circ m)\bigr)[\partial_j]
  &=
  \partial_i\!\left\{
    (\dd\gamma\circ m)[\dd m[\partial_j]]
  \right\}
  -
  (\dd\gamma\circ m)
  [D^W_{\partial_i}\partial_j]
  \\
  &=
  \partial_i\partial_j(\gamma\circ m)
  -
  (\dd\gamma\circ m)
  [D^W_{\partial_i}\partial_j].
\end{aligned}
\]
Equivalently,
\[
  \bigl(D^W(\dd\gamma\circ m)\bigr)
  [\partial_i,\partial_j]
  =
  \partial_i\partial_j(\gamma\circ m)
  -
  (\dd\gamma\circ m)
  [D^W_{\partial_i}\partial_j].
\]
Hence \eqref{eq:loss-Laplacian-definition} gives
\eqref{eq:decision-localized-expansion}, and therefore
\eqref{eq:loss-explicit-expansion} follows.
\end{proof}

\section{Proofs of the remaining results}
\label{sec:proof-remaining-results}

\begin{proof}[Proof of Proposition \ref{prop:likelihood-bartlett-compatibility}]
By \eqref{eq:expected-one-form} and
\eqref{eq:likelihood-diagonal-centering},
\[
  \E_{n,\theta_0}[\Psi_n(\theta_0)]=0
\]
as a covector. Hence \eqref{eq:beta-definition} holds with $\beta=0$.

As in Definition~\ref{def:precontrast},
\eqref{eq:likelihood-diagonal-centering} gives, for local vector fields
$X_1,X_2$,
\[
  \overline\Psi_n(X_2X_1\mid)(\theta_0)
  +\overline\Psi_n(X_1\mid X_2)(\theta_0)
  =0.
\]
Together with \eqref{eq:likelihood-score-covariance}, this yields
\begin{equation}
  a_n^{-2}\Cov_{n,\theta_0}\!\left\{
    \Psi_n[X_1](\theta_0),\Psi_n[X_2](\theta_0)
  \right\}
  =\overline\Psi_n(X_2X_1\mid)(\theta_0).
  \label{eq:likelihood-first-bartlett-identity}
\end{equation}

We next identify the right-hand side. Let
$x=(x^1,\ldots,x^{\mathrm p})$ be the coordinate system in
Assumption~\textup{[A1]}, and write $\partial_i:=\partial/\partial x^i$.
By \eqref{eq:second-jet-local-rate}, there exist $r>0$ and $n_0$ such that,
for every $i,j$,
\begin{equation}
  \sup_{n\ge n_0}
  \left\|
    \sup_{d(\theta,\theta_0)\le r}
    \left|
      \partial_i^2\{\Psi_n[\partial_j]\}(\theta)
    \right|
  \right\|_{L^1(P_{n,\theta_0})}
  <\infty.
  \label{eq:likelihood-local-second-derivative-bound}
\end{equation}
For $h$ sufficiently small, let $\theta_h$ be defined by
$x(\theta_h)=x(\theta_0)+he_i$. Taylor's theorem gives
\begin{equation}
  \Psi_n[\partial_j](\theta_h)
  =
  \Psi_n[\partial_j](\theta_0)
  +h\,\partial_i\{\Psi_n[\partial_j]\}(\theta_0)
  +h^2R_{n,ij}(h),
  \label{eq:likelihood-coordinate-Taylor}
\end{equation}
where, for all sufficiently small $h$,
\begin{equation}
  |R_{n,ij}(h)|
  \le
  \frac12
  \sup_{d(\theta,\theta_0)\le r}
  \left|
    \partial_i^2\{\Psi_n[\partial_j]\}(\theta)
  \right|.
  \label{eq:likelihood-coordinate-Taylor-remainder}
\end{equation}
By \eqref{eq:likelihood-local-second-derivative-bound} and
\eqref{eq:likelihood-coordinate-Taylor-remainder},
$hR_{n,ij}(h)\to0$ in $L^1(P_{n,\theta_0})$ as $h\to0$ for
$n\ge n_0$. Hence, by \eqref{eq:likelihood-coordinate-Taylor}, for $n\ge n_0$,
\[
\begin{aligned}
  \overline\Psi_n(\partial_i\partial_j\mid)(\theta_0)
  &=
  \lim_{h\to0}
  \E_{n,\theta_0}\!\left[
    \frac{
      \Psi_n[\partial_j](\theta_h)
      -\Psi_n[\partial_j](\theta_0)
    }{h}
  \right]
  \\
  &=
  \lim_{h\to0}
  \E_{n,\theta_0}\!\left[
    \partial_i\{\Psi_n[\partial_j]\}(\theta_0)
    +hR_{n,ij}(h)
  \right]
  \\
  &=
  \E_{n,\theta_0}\!\left[
    \partial_i\{\Psi_n[\partial_j]\}(\theta_0)
  \right].
\end{aligned}
\]
By \eqref{eq:first-jet-rate},
\[
  \overline\Psi_n(\partial_i\partial_j\mid)(\theta_0)
  \longrightarrow
  \braket{ \partial_i,\partial_j}(\theta_0).
\]
Therefore \eqref{eq:likelihood-first-bartlett-identity} gives, for every
$i,j$,
\[
  a_n^{-2}
  \Cov_{n,\theta_0}\!\left\{
    \Psi_n[\partial_i](\theta_0),\Psi_n[\partial_j](\theta_0)
  \right\}
  \longrightarrow
  \braket{ \partial_i,\partial_j}(\theta_0).
\]
Since both sides are tensorial in the values at $\theta_0$, it follows that,
for arbitrary local vector fields $X_1,X_2$,
\[
  a_n^{-2}
  \Cov_{n,\theta_0}\!\left\{
    \Psi_n[X_1](\theta_0),\Psi_n[X_2](\theta_0)
  \right\}
  \longrightarrow
  \braket{ X_1,X_2}(\theta_0).
\]
Thus \eqref{eq:covariance-normalization} holds.

Finally, for arbitrary local vector fields $X_1,X_2,X_3$,
\eqref{eq:likelihood-mixed-covariance} and
\eqref{eq:likelihood-mixed-jet} give
\[
\begin{aligned}
  a_n^{-2}
  \Cov_{n,\theta_0}\!\left\{
    X_1(\Psi_n[X_2])(\theta_0),
    \Psi_n[X_3](\theta_0)
  \right\}
  &\longrightarrow
  -\Psi(X_1X_2\mid X_3)(\theta_0)
  \\
  &=
  \braket{D_{X_1}^{(1)}X_2,X_3}(\theta_0),
\end{aligned}
\]
where the last equality is \eqref{eq:D1-definition}. Hence
\eqref{eq:DB-definition} holds with $D^{\mathrm B}=D^{(1)}$.
Together with $\beta=0$, this proves Assumption~\textup{[A3]} and, by
\eqref{eq:b-definition}, $b=0$.
\end{proof}

\begin{proof}[Proof of Proposition~\ref{prop:bias-expansion}]
Throughout this proof, a subscript $0$ denotes evaluation at $\theta_0$.
Choose local coordinates whose coordinate basis at $\theta_0$ is
$g_0$-orthonormal, and write $\Psi_{n,i}:=\Psi_n[\partial_i]$.
The definitions of $\zeta_n$ and $\Upsilon_n$ give
\[
  \zeta_{n,0}^i=-a_n^{-1}\Psi_{n,i,0},
  \qquad
  (\Upsilon_{n,0})^i{}_j
  =
  \delta^i_j-(\partial_j\Psi_{n,i})_0
  +\Gamma^{(1),k}_{ji,0}\Psi_{n,k,0}.
\]
By \eqref{eq:beta-definition} and \eqref{eq:first-jet-rate},
\[
  \E_{n,\theta_0}[\Psi_{n,i,0}]
  =
  a_n^2\beta_{i,0}+o(a_n^2),
  \qquad
  \E_{n,\theta_0}[(\partial_j\Psi_{n,i})_0]
  =
  \delta_{ji}+o(1).
\]
Moreover, \eqref{eq:DB-definition}, \eqref{eq:covariance-normalization},
and \eqref{eq:beta-definition} give
\begin{align*}
  a_n^{-2}\Cov_{n,\theta_0}
  \bigl((\partial_j\Psi_{n,i})_0,\Psi_{n,j,0}\bigr)
  &\longrightarrow
  \Gamma^{\mathrm B,j}_{ji,0},
  \\
  a_n^{-2}\E_{n,\theta_0}
  [\Psi_{n,k,0}\Psi_{n,j,0}]
  &\longrightarrow
  \delta_{kj},
\end{align*}
where the first relation is read for each fixed $j$ before summation.
Therefore
\begin{align}
  a_n^{-1}\E_{n,\theta_0}
  [\Upsilon_{n,0}[\zeta_{n,0}]^i]
  &=
  -a_n^{-2}\E_{n,\theta_0}[\Psi_{n,i,0}]
  +a_n^{-2}\sum_j
  \E_{n,\theta_0}[(\partial_j\Psi_{n,i})_0]
  \E_{n,\theta_0}[\Psi_{n,j,0}]
  \\
  &\quad
  + a_n^{-2}\sum_j
  \Cov_{n,\theta_0}
  \bigl((\partial_j\Psi_{n,i})_0,\Psi_{n,j,0}\bigr)
  -a_n^{-2}\sum_{j,k}\Gamma^{(1),k}_{ji,0}
  \E_{n,\theta_0}[\Psi_{n,k,0}\Psi_{n,j,0}]
  \\
  &\longrightarrow
  -\beta_{i,0}+\beta_{i,0}
  +\sum_j
  \bigl(
    \Gamma^{\mathrm B,j}_{ji,0}
    -
    \Gamma^{(1),j}_{ji,0}
  \bigr)
  =
  \sum_j
  \bigl(
    \Gamma^{\mathrm B,j}_{ji,0}
    -
    \Gamma^{(1),j}_{ji,0}
  \bigr).
  \label{eq:bias-proof-Upsilon-mean}
\end{align}
Moreover,
\begin{equation}
  a_n^{-1}\E_{n,\theta_0}[\zeta_{n,0}^i]
  \longrightarrow
  -\beta_{i,0}.
  \label{eq:bias-proof-zeta-mean}
\end{equation}
Equations~\eqref{eq:bias-proof-Upsilon-mean} and
\eqref{eq:bias-proof-zeta-mean}, together with
\eqref{eq:b-definition}, give
\eqref{eq:scaled-jet-first-moment}.

Likewise, \eqref{eq:covariance-normalization} and
\eqref{eq:beta-definition} give
\[
  \E_{n,\theta_0}
  [\zeta_{n,0}^i\zeta_{n,0}^j]
  \longrightarrow
  \delta_{ij}.
\]
Since the coordinates are $g_0$-orthonormal, $g_0^{ij}=\delta^{ij}$,
and hence \eqref{eq:scaled-jet-second-moment} follows.
\end{proof}

\begin{proof}[Proof of Corollary~\ref{cor:loss-temperature-drift-comparison}]
Subtract the two expansions in Theorem~\ref{thm:main-expansion} and use
\eqref{eq:estimator-comparison-operator}. For the decision rules, the
common target map $m$ makes the difference of the expansion terms equal to
$(\mathcal M_n^{(1)}-\mathcal M_n^{(2)})(\gamma\circ m)$, while the two
loss-Laplacian terms remain. This proves both stochastic comparisons.
\end{proof}

\begin{proof}[Proof of Corollary~\ref{cor:coordinate-expansion}]
Under Setting~\textup{[E]}, the standard coordinate map satisfies
\[
  \Delta\theta=-\Gamma^\nabla[g^{-1}],
  \qquad
  (D^{(-1)}\dd\theta)[S]=-\Gamma^{(-1)}[S].
\]
Applying Theorem~\ref{thm:main-expansion} to this map gives
\eqref{eq:euclidean-parameter-expansion}, and
Corollary~\ref{cor:loss-temperature-drift-comparison} gives
\eqref{eq:estimator-comparison-coordinate}. Finally,
\eqref{eq:general-bias-formula} yields
\[
  \calB\theta
  =
  V-b-\left\{
    \tau\Gamma^\nabla+\frac{1-\tau}{2}\Gamma^{(-1)}
  \right\}[g^{-1}],
\]
which proves \eqref{eq:coordinate-bias-expansion} and
\eqref{eq:coordinate-second-order-unbiased}.
\end{proof}

\begin{proof}[Proof of Corollary~\ref{cor:bias-comparison-correction}]
For any connection $D$ on $T\Theta$, the dual derivative rule gives
\[
  \Delta^D(\gamma\circ m)-\Delta^W\gamma
  =(\dd\gamma\circ m)
       [(D^W-\dd m\circ D)[g^{-1}]].
\]
Apply this with $D=\nabla$ and $D=D^{(-1)}$ in
\eqref{eq:general-bias-formula}. By \eqref{eq:loss-calibration-section},
\begin{equation}
  \calB(\gamma\circ m)-\frac{1+\tau}{2}\Delta^W\gamma
  =(\dd\gamma\circ m)[\dd m[V]-V^*].
  \label{eq:loss-calibration-identity}
\end{equation}
Subtract the two formulas in Theorem~\ref{thm:main-expansion}, with
$f=\gamma\circ m$. Since $\tau>0$ is fixed,
\[
  \widehat{\mathbb B}_n(\gamma)
  =\frac12\Delta^W\gamma(\theta_0)
    +o_{\Linftyminus(\theta_0)}(1).
\]
Proposition~\ref{prop:bias-expansion} gives
\[
  \mathbb B_{n,\theta_0}(\gamma)
  =\left\{\calB(\gamma\circ m)-\frac\tau2\Delta^W\gamma\right\}(\theta_0)
    +o(1).
\]
Subtracting and using \eqref{eq:loss-calibration-identity} proves
\eqref{eq:normalized-bias-comparison}. Finally, by the definitions,
\[
  \E_{n,\theta_0}[\widetilde\gamma_n]-\gamma(m(\theta_0))
  =a_n^2\E_{n,\theta_0}
      [(\mathbb B_{n,\theta_0}-\widehat{\mathbb B}_n)(\gamma)].
\]
The remainder in \eqref{eq:normalized-bias-comparison} is $o_{L^1}(1)$.
Thus $\dd m[V]=V^*$ yields \eqref{eq:loss-corrected-bias-expansion}.
\end{proof}

\begin{proof}[Proof of Corollary~\ref{thm:canonical-loss-calibrations}]
For the squared-distance loss in \eqref{eq:distance-loss-geometry},
$m=\operatorname{id}_\Theta$ and $D^W=\nabla^h$. For any fixed $\tau>0$,
\eqref{eq:loss-calibration-section} therefore gives
\begin{equation}
  V^*=b+\left\{
       \tau(\nabla-\nabla^h)
       +\frac{1-\tau}{2}(D^{(-1)}-\nabla^h)
     \right\}[g^{-1}].
  \label{eq:distance-calibrating-drift-temperature}
\end{equation}
In particular, if $h=g$, then
\begin{equation}
  V^*=b+\frac{\tau-1}{2}(\nabla-D^{(-1)})[g^{-1}].
  \label{eq:metric-calibrating-drift-temperature}
\end{equation}
At $\tau=1$, this is $V^*=b$. Since
$m=\operatorname{id}_\Theta$, the choice $V=b$ in \textup{(i)} satisfies
\eqref{eq:loss-calibrating-drift}.

In \textup{(ii)}, $m=\operatorname{id}_\Theta$, $D^W$ is the flat
coordinate connection, and $\tau=1$. Hence, in the standard coordinates,
\begin{equation}
  V^*=b+\Gamma^\nabla[g^{-1}].
  \label{eq:unit-correcting-vector-field}
\end{equation}
Thus $V=b+\Gamma^\nabla[g^{-1}]$ satisfies
\eqref{eq:loss-calibrating-drift}.
\end{proof}

\appendix

\section{Godambe Normalization of Estimation Functions}
\label{app:godambe-normalization}

We record the standard Godambe normalization used in
Sections~\ref{sec:setup} and~\ref{subsec:local-assumptions}; see
Godambe~\cite{Godambe1960}.

\begin{proposition}[Godambe normalization]
\label{prop:godambe-normalization}
Let $\Psi_{n,0}$ be a random smooth $\R^{\mathrm p}$-valued estimation
function on $\Theta$. Suppose that there exist smooth families
\[
  H(\theta)\in\operatorname{Hom}(T_\theta\Theta,\R^{\mathrm p}),
  \qquad
  J(\theta)\in\operatorname{Sym}^2(\R^{\mathrm p}),
\]
such that $H(\theta)$ is an isomorphism and $J(\theta)$ is positive definite
for every $\theta\in\Theta$. Assume, for each fixed $\theta\in\Theta$,
\begin{equation}
  \Psi_{n,0}(\theta)=O_{P_{n,\theta}}(a_n),
  \qquad
  (\dd\Psi_{n,0})_\theta
  \longrightarrow H(\theta)
  \label{eq:godambe-sensitivity-assumptions}
\end{equation}
in $P_{n,\theta}$-probability, and
\begin{equation}
  a_n^{-2}\Cov_{n,\theta}
  \!\left(\Psi_{n,0}(\theta)\right)
  \longrightarrow J(\theta).
  \label{eq:godambe-variability-assumption}
\end{equation}
Regard $J(\theta)$ as the isomorphism
$J(\theta):(\R^{\mathrm p})^*\to\R^{\mathrm p}$ and define
\begin{equation}
  \Psi_n(\theta)
  :=
  H(\theta)^*J(\theta)^{-1}\Psi_{n,0}(\theta)
  \in T_\theta^*\Theta.
  \label{eq:godambe-normalized-one-form}
\end{equation}
Then, for any connection $D$ on $T\Theta$ and its dual connection on
$T^*\Theta$,
\begin{equation}
  (D\Psi_n)_\theta
  \longrightarrow
  H(\theta)^*J(\theta)^{-1}H(\theta)
  \label{eq:godambe-normalized-sensitivity}
\end{equation}
in $P_{n,\theta}$-probability, and
\begin{equation}
  a_n^{-2}\Cov_{n,\theta}\!\left(\Psi_n(\theta)\right)
  \longrightarrow
  H(\theta)^*J(\theta)^{-1}H(\theta).
  \label{eq:godambe-normalized-variability}
\end{equation}
The common limit is
\[
  H(\theta)^*J(\theta)^{-1}H(\theta)
  \in\operatorname{Sym}^2(T_\theta^*\Theta).
\]
\end{proposition}

\begin{proof}
Set $A:=H^*J^{-1}$, viewed as a smooth bundle homomorphism from the trivial
bundle $\Theta\times\R^{\mathrm p}$ to $T^*\Theta$. By the product rule,
\[
  D\Psi_n=(DA)\Psi_{n,0}+A\,\dd\Psi_{n,0}.
\]
At each fixed $\theta$, the first term is
$O_{P_{n,\theta}}(a_n)=o_{P_{n,\theta}}(1)$, while the second converges in
probability to
$H(\theta)^*J(\theta)^{-1}H(\theta)$, proving
\eqref{eq:godambe-normalized-sensitivity}. Since $A(\theta)$ is
deterministic,
\[
\begin{aligned}
  a_n^{-2}\Cov_{n,\theta}\!\left(\Psi_n(\theta)\right)
  &=
  A(\theta)
  \left\{
    a_n^{-2}\Cov_{n,\theta}\!\left(\Psi_{n,0}(\theta)\right)
  \right\}
  A(\theta)^*\\
  &\longrightarrow
  H(\theta)^*J(\theta)^{-1}H(\theta),
\end{aligned}
\]
which proves \eqref{eq:godambe-normalized-variability}.
\end{proof}

\end{document}